\documentclass[utf8]{interact}

\usepackage{epstopdf}% To incorporate .eps illustrations using PDFLaTeX, etc.

\usepackage[utf8]{inputenc}
\usepackage{lscape}
\usepackage{mathtools}
\usepackage{commath}
\usepackage{amsmath}
\usepackage{amsthm}
\usepackage{amssymb}
\usepackage{dsfont}
\usepackage{bm}
\usepackage{graphicx}
\usepackage{xcolor}
\usepackage[caption=false]{subfig}
\usepackage{booktabs} 
\usepackage{multirow}
\usepackage{pgfplots}
\usepackage{listings}
\usepackage{url}
\usepackage{fouriernc}
\usepackage{IEEEtrantools}
\usepackage{enumitem}

\DeclareMathAlphabet{\pazocal}{OMS}{zplm}{m}{n}
\DeclareSymbolFont{greekletters}  {OML}{ztmcm}{m}{it}

\DeclareMathSymbol{\chi}{\mathalpha}{greekletters}{"1F}
\DeclareMathOperator*{\esssup}{ess\,sup}

\newcommand{\ind}{\mathds{1}}

\newcommand{\ew}{\mathds{E}}
\newcommand{\pr}{\mathds{P}}

\newcommand{\Rbb}{\mathbb{R}}

\newcommand{\Nbb}{\mathbb{N}}
\newcommand{\Tbb}{\mathbb{T}}
\newcommand{\Zbb}{\mathbb{Z}}
\newcommand{\var}{\mathrm{var}}
\newcommand{\cov}{\mathrm{cov}}
\newcommand{\cf}{\mathcal{F}}
\newcommand{\cg}{\mathcal{G}}

\newcommand{\cm}{\mathcal{M}}

\newcommand{\ce}{\mathcal{E}}

\newcommand{\cb}{\mathcal{B}}

\newcommand{\ch}{\mathcal{H}}

\newcommand{\cs}{\mathcal{S}}

\def\1{{\ind}}
\def\argmin{\mbox{arg}\min}

\usepackage{natbib}% Citation support using natbib.sty
\bibpunct[, ]{(}{)}{; }{a}{,}{,}% Citation support using natbib.sty
\makeatletter% @ becomes a letter
\def\NAT@def@citea{\def\@citea{\NAT@separator}}% Suppress spaces between citations using natbib.sty
\makeatother% @ becomes a symbol again

\theoremstyle{plain}% Theorem-like structures provided by amsthm.sty
\newtheorem{theorem}{Theorem}[section]
\newtheorem{lemma}[theorem]{Lemma}
\newtheorem{folg}[theorem]{Corollary}

\theoremstyle{definition}
\newtheorem{defi}[theorem]{Definition}

\theoremstyle{remark}
\newtheorem{rem}{Remark}

\begin{document}

\articletype{ARTICLE MANUSCRIPT}% Specify the article type or omit as appropriate

\title{Consistency of a nonparametric least squares estimator in integer-valued GARCH models}

\author{
\name{Maximilian Wechsung\textsuperscript{a,*}\thanks{\textsuperscript{*} Corresponding author. Email: maximilian.wechsung@charite.de}, Michael Neumann\textsuperscript{b}}
\affil{\textsuperscript{a}Institut f\"ur Biometrie und Klinische Epidemiologie, Charit\'e--Universit\"atsmedizin Berlin, \\ \, Charit\'eplatz 1, 10117 Berlin, Germany \\
\textsuperscript{b}Institut f\"ur Mathematik, Friedrich-Schiller-Universit\"at Jena \\ \,
Ernst-Abbe-Platz 2, 07743 Jena, Germany}
}

\maketitle

\begin{abstract}
We consider a nonparametric version of the integer-valued GARCH(1,1) model for time series of counts. The link function in the recursion for the variances is not specified by finite-dimensional parameters, but we impose nonparametric smoothness conditions. We propose a least squares estimator for this function and show that it is consistent with a rate that we conjecture to be nearly optimal.  \\[5pt]
\noindent \textit{Number of words}: 7,664
\end{abstract}

\allowdisplaybreaks

\section{Introduction}

Time dependent count data appear in many branches of empirical research.
In an epidemiological context, count data in form of reported infections in a certain time interval are currently attracting considerable attention. 

A common model for the marginal conditional distributions of a time series of count variables
$\{Y_t\}_{t \in \Zbb}$ is the Poisson distribution, i.e.
$Y_t \, | \, Y_{t-1},Y_{t-2},\ldots \sim \text{Poiss}(\lambda_t)$
\citep{Kedem2002}.
Here the intensities $\{\lambda_t\}_{t \in \Zbb}$ are mere theoretical quantities defined
for modeling purposes. As such they are not observable. A temporal dynamic can be modeled
with a recursive relation for the intensities, $\lambda_t = m(Y_{t-1},\ldots,Y_{t-p};\lambda_{t-1}, \ldots,\lambda_{t-q})$
for some $p,q \in \Nbb_+$. This relation could as well include some exogenous variables on the right hand side.
However, in order to focus on the inherent dynamics of the process, we forgo the inclusion of such explanatory variables.
The function $m$ is called link function and is at the center of our interest.

The introduced Poisson autoregression model is very similar to the GARCH($p,q$) model that was introduced by
\cite{Bollerslev1986} and is constituted by a process $\{X_t\}$ with
$X_{t} | X_{t-1},X_{t-2},\ldots \sim N(0, \sigma_t^2)$ and
$\sigma_t^2 = \alpha_0 + \alpha_1 X_{t-1}^2 + \ldots \alpha_p X^2_{t-p} + \beta_1 \sigma_{t-1}^2 + \ldots + \beta_q \sigma_{t-q}^2$. Therefore, our Poisson model is often called integer-valued GARCH($p,q$) or INGARCH$(p,q)$ model.
This term was introduced by \cite{Ferland2006}, who showed existence and stationarity of an INGARCH$(p,q)$ process with linear link function.
Parameter estimation in a linear INGARCH(1,1) model was considered by \cite{Fokianos2009}, who proved consistency and asymptotic normality of the conditional maximum likelihood estimator. 

Both the structural as well as the distributional assumptions imposed by linear INGARCH models are rather limiting.
To overcome these impediments, two intuitive approaches seem appropriate. First, we could alter the distributional
assumption so as to allow for more flexibility in fitting the model. This approach was chosen by \cite{Zhu2011, Zhu2012}. He introduced count models based on the more flexible negative binomial
and generalized Poisson distributions, respectively. The structural components in these models, however, still resemble that of a linear INGARCH($p,q$) process. 

The alternative approach is to weaken the structural model assumption so as to allow for more general link functions.
For instance, \cite{Fokianos2011} extended their earlier result to INGARCH(1,1) models with link functions of the form
$m(y,\lambda) = f(\lambda) + g(y)$ for rather general but still parametric functions $f$ and $g$.
To our knowledge, the first published results on a purely nonparametric INGARCH(1,1) model is due to \cite{Neumann2011}. He proved absolute regularity of the count process $\{Y_t\}$ under the assumption that the
link function $m$ satisfies a strong contraction property. \cite{Doukhan2017} proved absolute regularity of general INGARCH$(p,q)$ processes, assuming only a semi-contractive link function. 

In the present paper, we follow the second approach to model generalization.  
So far the problem of estimating a nonparametric link function has not been addressed yet.
We propose and analyse a theoretical least squares estimator for the link function in a nonparametric INGARCH(1,1) model,
presupposing the same contraction property as \cite{Neumann2011}. Statistical inference in nonparametric GARCH$(1,1)$ models was considered by \cite{Meister2016}---the similarity between the GARCH and INGARCH models motivated us to mimic their estimator.
However, we pursue a different strategy in the asymptotic analysis which will be based on mixing properties. 
At the end we will briefly discuss our main theorem in light of related minimax results. Additionally we will show how an approximate version of our theoretical estimator can be implemented and illustrate that with a small simulation study.

\section{Assumptions and main result}

\begin{defi}\label{def:model}
For $\mathbb{T} \in \{ \Nbb, \Zbb\}$, let $\{(\lambda_t,Y_t)\}_{t \in \Tbb}$ be a stochastic process that is defined
on a probability space $(\Omega, \cf, \pr)$ and assumes values in $\Rbb \times \Nbb = \Rbb \times \{0,1,2, \ldots\}$.
Let $\cf_{t} := \sigma \{ \lambda_s, Y_s \colon s \leq t \}$ be the $\sigma$-field generated by the process up to time $t$,
and let $\cb$ denote the Borel $\sigma$-field over $\Rbb$. The bivariate process $\{(\lambda_t,Y_t)\}_{t \in \Tbb}$ is called
an \textit{INGARCH(1,1)} process if there exists a $(\cb \otimes 2^\Nbb - \cb)$-measurable function
$m \colon [0,\infty) \times \Nbb \to [0,\infty)$ such that 
\begin{IEEEeqnarray*}{rCL}
\pr^{Y_t | \cf_{t-1}} = \textit{Poiss}(\lambda_t) 
& \hspace{1 cm} \text{ and } \hspace{1 cm} 
& \lambda_t = m(\lambda_{t-1},Y_{t-1})\,.
\end{IEEEeqnarray*} 
The processes $\{Y_t\}$ and $\{\lambda_t\}$ are called \textit{count process} and \textit{intensity process} respectively. The function $m$ is called \textit{link function}.
\end{defi}

Given a link function $m$, the corresponding one-sided INGARCH(1,1) process (i.e. with $\mathbb{T} = \Nbb$) is well defined.
It is a Markov process and can be constructed explicitly by specifying the transition kernel.
A two-sided version (i.e. $\mathbb{T} = \Zbb$) is well defined if a stationary distribution of the one-sided process exists \citep{Doukhan2017, Wechsung2019}.   

Recall that the intensities are hidden variables. A prerequisite for estimating the link function without knowledge
of the intensities is that a single intensity $\lambda_t$ does not exclusively carry substantial information about
the whole process. In other words, we require the influence of $\lambda_t$ to fade with progressing time.
This property can be compelled by stipulating that the link function be contractive.
 	 
\begin{defi}
\label{def:semicont}
For $M<\infty$
and numbers $L_1,L_2 \geq 0$ such that
$L_1 + L_2 < 1$, the class $\cg = \cg(M,L_1,L_2)$ of contractive link functions with domain $D:=[0,M]\times \Nbb$
and smoothness parameters $L_1$ and $L_2$ is defined as the set of all functions $g \colon D  \to [0,M]$
with the property that
\begin{align}\tag{C} 
|g(\lambda_1,y_1) - g(\lambda_2,y_2)| \leq L_1 |\lambda_1 - \lambda_2| + L_2 |y_1-y_2|\,,
\end{align}
for all $(\lambda_1,y_1), (\lambda_2,y_2) \in D$. 
\end{defi}
In case of a contractive link function $m \in \cg(M,L_1,L_2)$, the influence of a single intensity variable
indeed vanishes with time:
define inductively 
$m^{[0]}(\lambda,Y_t) := m(\lambda,Y_t)$ and
$m^{[k]}(\lambda,Y_t,\ldots,Y_{t+k}) := m(m^{[k-1]}(\lambda,Y_t,\ldots,Y_{t+k-1}), Y_{t+k})$.
Observe that due to (C)
\begin{align*}
\big| \lambda_{t+1} \,-\, m^{[0]}(0,Y_t) \big|
\,=\, \big| m(\lambda_t,Y_t) \,-\, m(0,Y_t) \big| \,\leq\, L_1 \, \lambda_t
\end{align*}
and therefore
\begin{IEEEeqnarray}{rCl}
\IEEEeqnarraymulticol{3}{l}{\big| \lambda_{t+k} \,-\, m^{[k-1]}(0,Y_t,\ldots ,Y_{t+k-1}) \big|   \nonumber } \\
 \qquad 
& = & \big| m^{[k-1]}(\lambda_t,Y_t,\ldots,Y_{t+k-1}) \,-\, m^{[k-1]}(0,Y_t,\ldots,Y_{t+k-1}) \big| \nonumber \\
& = & \Big| m\big( m^{[k-2]}(\lambda_t,Y_t,\ldots,Y_{t+k-2}), Y_{t+k-1} \big)
\,-\, m\big( m^{[k-2]}(0,Y_t,\ldots,Y_{t+k-2}), Y_{t+k-1}\big) \Big| \nonumber \\
& \leq & L_1 \, \big| m^{[k-2]}(\lambda_t,Y_t,\ldots,Y_{t+k-2}) \,-\, m^{[k-2]}(0,Y_t,\ldots,Y_{t+k-2}) \big| \nonumber \\
& \leq & \ldots \,\leq\, L_1^k \, \lambda_t \to 0 
\label{C2}
\end{IEEEeqnarray}
in probability as $k \to \infty$. 
This also ensures the existence of a stationary distribution $\pi$ because two one-sided INGARCH(1,1) processes with the same link function but different starting points would eventually behave alike, indicating that they have reached a stationary regime 
\citep{Neumann2011, Doukhan2017}. Hence, the two-sided INGARCH(1,1) process is also well defined 
\citep{Doukhan2017, Wechsung2019}. Furthermore, equation \eqref{C2} implies $\cf_t = \sigma\{Y_t,Y_{t-1},\ldots\}$, which means that all necessary information for statistical inference is carried by the count process 
\citep{Neumann2011}. Thus, from a statistical point of view, estimating the link function without observing any intensity should be a feasible task. 

We propose a minimum contrast estimator for the link function.
Recall that the functional $X \mapsto \ew\big(Y_{n+1} - X \big)^2$ is minimized on the set of all $\sigma\{Y_0,\ldots,Y_n\}$-measurable random variables by $X = \ew[Y_{n+1} \, | \, Y_0,\ldots,Y_n]$. Furthermore, the contraction property (C) implies $\big|\ew[Y_{n+1} \, | \, Y_0,\ldots,Y_n] - m^{[n]}(0,Y_{0}, \ldots, Y_n)\big| = O(L_1^n)$ for $n \to \infty$. Hence, $m$ is an approximate minimum of the contrast functional $\Phi : g \mapsto \ew\big(Y_{i+1} - g^{[i]}(0,Y_{0}, \ldots, Y_i)\big)^2$ over $\cg$, and a sensible estimator on the basis of observations $Y_0,\ldots,Y_n$ might be obtained by minimizing an empirical analogue of that functional.  

\begin{defi}\label{defi:estimator}
Denote with $\cb(\cg)$ the Borel $\sigma$-field over the normed space $(\cg,\|\cdot\|_\infty)$,
and let $\cg^{\textit{cand}}[Y_0,\ldots,Y_n] \subset (\cg,\|\cdot\|_\infty)$ be a non-void closed and possibly random subset of candidate functions.
A random element $\hat{m}_n[Y_0,\ldots,Y_n] \colon (\Omega,\cf) \to (\cg,\cb(\cg))$ that minimizes the empirical contrast functional $\Phi_n(g)\colon g \mapsto \frac{1}{n} \sum_{i=0}^{n-1} \big(Y_{i+1} - g^{[i]}(0,Y_0, \ldots, Y_i)\big)^2
$ over $\cg^{\textit{cand}}$ is called a least squares estimator based on $n+1$ consecutive observations of the count process and the candidate set $\cg^{\textit{cand}}$.
\end{defi}

The set of least squares estimators is non-void. To prove this assertion, it has to be shown that, given realizations of the count process, the minimum of $\Phi_n(g)$ over $\cg^{\textit{cand}}$ is attained and that there exists a
$(\cb(\Rbb^{n+1})-\cb(\cg))$-measurable selection function $T$ assuming values in the set of the functional's minimizers. The random element $T(Y_0,\ldots,Y_n)$ would then qualify as a least squares estimator.
The functional $\Phi_n$ attains its minimum over $(\cg^{\textit{cand}},\|\cdot\|_\infty)$ since it is continuous with respect to $\|\cdot\|_\infty$ and because the space $(\cg,\|\cdot\|_\infty)$ is compact. The existence of a measurable selector can be proven by a successive application of Jennrich's Lemma 
\citep{jennrich1969} and the Kuratowski-Ryll-Nardzewski selection theorem 
\citep[cf.][]{aliprantis2007}. The formal argument is rather technical and will be omitted here. For details we refer to \cite{Wechsung2019}.

Before we present our main result on the asymptotic behaviour of the least squares estimator, we summarize our technical assumptions. 
\begin{itemize}
\item[(A1)]
For $M<\infty$ and non-negative constants $L_1,L_2$ with $L_1+L_2<1$, let the data generating process $\{(Y_t,\lambda_t)\}_{t \in \Zbb}$
be a stationary INGARCH(1,1) process with link function $m \in \cg(M,L_1,L_2)$ and stationary distribution $\pi = \pr^{(\lambda_t,Y_t)}$.

\item[(A2)] Denote with $Y_{(n)} = \max \{Y_0, \ldots, Y_{n-1}\}$ the largest of the first $n$ observations, and define the random set
\begin{align*}
\cg_n^{\textit{cand}}(M,L_1,L_2) := \big\{ g \in \cg(M,L_1,L_2) \colon  g(\, \cdot\,, y)=g(\, \cdot\,, Y_{(n)}) \text{ for all } y > Y_{(n)} \big\}\,.
\end{align*}
The estimator $\hat{m}_n$ shall be a least squares estimator in the sense of Definition \ref{defi:estimator} with candidate set $\cg_n^{\textit{cand}}(M,L_1,L_2)$.
\end{itemize}

Regarding assumption (A2), note that the maximum $Y_{(n)}$ does not grow faster than $\log n$ because due to an exponential tail bound for the Poisson distribution \citep[p. 116]{gine2015} 
\begin{align*}
\pr\big\{\max_{0\leq i \leq n-1} Y_i \geq C \log n\big\} 
&\leq n \, \pr \big\{ Y_0-\lambda_0 \geq C \log n - \lambda_0\big\} \\
&\leq n \, \ew \Big[ \pr \big\{ Y_0-\lambda_0 \geq C \log n - \lambda_0 \; \big| \; \lambda_0 \big\} \Big] \\
&\leq n \, \ew \Big[ \exp\Big(- \frac{(C \log n - \lambda_0)^2}{2\lambda_0 + \frac{2}{3}(C \log n - \lambda_0)} \Big) \Big] \\
&\leq \exp\big((1 - C/2	) \log n + M \big) \to 0\,,
\end{align*}
as $n \to \infty$, if $C > 2$. This means that for $B_n = \lceil 3\, \log n \rceil$
\begin{align}\label{eq:cand}
\pr \big\{ \cg_n^{\textit{cand}}(M,L_1,L_2) \subset  \cg(B_n,M,L_1,L_2)\big\} \to 1\,,
\end{align}
as $n \to \infty$, where $\cg(B_n,M,L_1,L_2) := \big\{g \in \cg(M,L_1,L_2) \colon g(\,\cdot\,,y) = g(\,\cdot\,,B_n) \text{ for all } y \geq B_n \big\}$.

The following theorem is our main result. It establishes the rate of convergence of the nonparametric least squares estimator in terms of the $L_2$ loss with respect to the stationary distribution $\pi$. 

\begin{theorem}\label{thm:main}
Suppose (A1) and (A2) hold. Then the $L_2(\pi)$ loss of the estimator, $L(\hat{m}_n,m) = \int \big(\hat{m}_n(\lambda,y) - m(\lambda,y)\big)^2 \pi(d\lambda,dy)$, is of order $O_\pr\big(n^{-2/3} (\log n)^2 \big)$.
\end{theorem}

The proof of this theorem is given in Section \ref{sec:proof_main}. It will proceed along a series of auxiliary results. For the sake of clarity, the proofs of most auxiliary results are deferred to Section \ref{sec:proof_aux}. In the next two sections we will shortly discuss the theoretical and practical implications of Theorem \ref{thm:main}.

\section{Minimax rates}\label{sec:minimax}

The rate $L(\hat{m}_n,m) = O_\pr(n^{-2/3} (\log n)^2)$ reflects the size of the function class $\cg$ in terms of the smoothness of the candidate functions and the dimension of their domains. In nonparametric regression and density estimation with i.i.d. samples, lower bounds for the squared $L_2$-loss of estimators for functions with degree of smoothness $\beta$, in terms of a Sobolev parameter, and domain dimension $d$ are typically given by $n^{-2\beta / (2\beta + d)}$, up to a positive multiplicative constant 
\citep{Tsybakov2008}.
In our model, the discrete nature of the count component and the slow growth of $Y_{(n)} = O_\pr(B_n)$ essentially reduce the problem to a parametric one in the second component. Thus, the effective dimension of the nonparametric estimation problem is $d=1$.
As to the degree of smoothness, the contractive property is a tighter version of Lipschitz continuity. Consequently, $\int_{[0,M]} \big|\frac{\partial}{\partial \lambda} m(\lambda,y_0)\big|^2 \ d\lambda < \infty$ for any $y_0 \in \{0,\ldots,B_n\}$
\citep[cf.][]{Bass2013}, i.e. $m(\,\cdot\,, y)$ is an element of the Sobolev class of functions with smoothness parameter $\beta = 1$ 
\citep[cf.][]{Tsybakov2008}. Hence, we suspect the optimal rate in our INGARCH(1,1) model to be of order $n^{-2/3}$. 

This conjecture is corroborated	by the following result for a nonparametric GARCH(1,1) model. Suppose the corresponding link function $m$ belongs to a Hölder class of monotone functions with smoothness parameter $\beta$ and domains with dimension $d=2$. Imposing some additional regularity conditions,
\cite{Meister2016} proved that in this case 
\begin{align}
\inf_{\{\hat{m}_n\}} \ \liminf_{n \to \infty} \  n^{2\beta / (2 \beta + d)} \sup_{m} \ \ew \int \big(m-\hat{m}_n\big)^2 \ d\pi > 0 \,.
\end{align} 
There, $\pi$ denotes the stationary distribution of the data generating GARCH$(1,1)$ process, and the infimum is taken over all sequences of estimators. 

In conclusion, we strongly suspect that the rate we provide in Theorem \ref{thm:main} is optimal up to the logarithmic term.

\section{Practical estimation and computer experiments}\label{sec:practical}

In light of the account of section \ref{sec:minimax}, our main result encourages to use a least squares approach to estimate the link function in a nonparametric INGARCH(1,1) model. One possibility to translate the theoretical concept of a nonparametric least squares estimator into a feasible procedure is via the methods of sieves.  
The sieved least squares estimator $\tilde{m}_n$ minimizes the functional $\Phi_n$ over a finite subset $\tilde{\cg}_n^{\textit{cand}} \subset \cg_n^{\textit{cand}}$. If the sequence of these finite candidate sets satisfy
$\sup_{g \in \cg_n^{\textit{cand}}} \inf_{g_n \in \tilde{\cg}_n^{\textit{cand}}} \| g - g_n\|_\infty \asymp n^{-1/3}$, we can expect $\tilde{m}_n$ to have asymptotic properties similar to $\hat{m}_n$ \citep{gyorfi2002}. 
We present a possible construction based on B-splines.

Let the equidistant knot sequence $\{\xi_p\}_{p=-2}^{l_n+2}$ be given by
\begin{align*}
\xi_{-2} < \xi_{-1} < \xi_0 = 0 < \xi_1 \ldots < \xi_{l_n} = M < \xi_{l_n+1}  
< \xi_{l_n+2}\,.
\end{align*} 
The knot distance is denoted by $\Delta_n := \xi_{i+1}-\xi_i$ for $i\in \{-2, \ldots,l_n+1\}$. Let $A_n$ be an equidistant partition of $[0,M]$ with $\frac{l_n}{2}+1 \leq \# A_n \leq C \, l_n$ points; $A_n = (\alpha_0,\alpha_1, \ldots, \alpha_{K_n})$ in ascending order with $a_0 = 0$ and $a_{K_n} = M$. The sequence 
$\{\Delta_n\}$ is required to satisfy $\Delta_n \asymp n^{-1/3}$. Note that then $l_n = \frac{M}{\Delta_n} \asymp n^{1/3}$. Let $\cs(k;\xi_{-k}, \ldots, \xi_{l+k}; A_n)$ denote the linear combinations of the normalized B-splines $\big\{ N_{p,k}|_{[0,M]} \big\}_{p=-k}^{l-1}$ with coefficients $\alpha_p$ from the set $A_n$. The grid $\tilde{\cg}_{n}^{\textit{cand}}$ defined by
\begin{align*}
\tilde{\cg}_{n}^{\textit{cand}} := \big\{s \in \cg_n^{\textit{cand}} \colon s(\,\cdot \, , y) \in \cs(2;\xi_{-2}, \ldots, \xi_{l_n+2} ; A_n) \ \textit{ for all } y \in \{0, \ldots , Y_{(n)}\} \big\} 
\end{align*}
approximates $\cg_n^{\textit{cand}}$ sufficiently well \citep{Wechsung2019}. For $\bm{\alpha} = \big(\alpha_{-2}(y) , \ldots, \alpha_{l_n-1}(y) \big)_{ y = 0, \ldots, Y_{(n)}} \in A_n^{(l_n+2)Y_{(n)}}$ and $S_{n} \bm{\alpha}\, [\lambda,y] = \sum_{p=-2}^{l-1} \alpha_p(y) \, N_{p,2}(\lambda)$, we can compute the estimator $\tilde{m}_n$ by minimizing the functional 
\begin{align}\label{eq:opt}
\bm{\alpha} \mapsto \sum_{i=0}^{n-1} \Big(y_{i+1} - \big(S_{n}\bm{\alpha}\big)^{[i]} [0, y_{0}, \ldots , y_i] \Big)^2
\end{align}
under the restriction that $S_{n}\bm{\alpha} \in \cg_n^{\textit{cand}}$.

We illustrate the sieved least squares estimator in a series of computer simulations. Having defined a link function
\begin{align*}
m(\lambda,y) = \Big( a + b \, \lambda + c \, \big(y \wedge 5 \big) + d \,\big[ \sin \left(\frac{2\pi}{\nu} \, \lambda\right) + \cos\left(\frac{2\pi}{\nu} \, \big( y \wedge 5 \big) \right) \big] \Big) \wedge M 
\end{align*}     
with $M=2$, $a = b = c = 0.3$, $d=-0.1$, and $\nu = 2$, we simulated 1050 realizations of the corresponding INGARCH(1,1) process, discarding the first 50 realizations in order to reach an approximately stationary regime. The link function $m$ and the simulated count observations are depicted in Figure \ref{fig:1}. To specify the set of candidate functions, we chose the underlying knot sequence to be $\{-0.4, -0.2,0.0, \ldots, 2.0, 2.2, 2.4\}$. The set $A_n$ was chosen to be an equidistant partition of the interval $[0,2]$ into $20$ parts. We approximately solved the resulting global optimization problem using the Genetic Algorithm (GA), a Monte Carlo method for global optimization problems, implemented in form of the MATLAB\textsuperscript{\textregistered} (version R2018a) function \texttt{ga()}. We applied the algorithm 4 times, yielding the 4 independent estimations that are depicted in Figure \ref{fig:2}. For the implementation we used MATLAB\textsuperscript{\textregistered} version R2018a. The code is available from the authors upon request.

\section{Conclusion}

We have shown that in principle a sufficiently regular link function of an INGARCH(1,1) process with hidden intensities can be estimated using a nonparametric least squares estimator. Furthermore, we have good reason to conjecture that this strategy is asymptotically nearly optimal in the minimax sense. For the purpose of practical implementation, we proposed to use a sieved version of our theoretical estimator and demonstrated the principle feasibility of this approach in numerical experiments. In our view, there is still a large potential to enhance the quality of this estimation method by improving the solver of the global optimization problem \eqref{eq:opt}. This is certainly a complicated problem itself, which we therefore leave for future research.

Essential for our theoretical arguments are of course the properties of the Poisson distribution. We settled for this distributional assumption because, in our perception, models based on the Poisson distribution still predominate the literature on count data. However, we acknowledge the associated practical drawbacks and deem it most desirable to combine our methods with more flexible distributional assumptions like those proposed by \cite{Zhu2011, Zhu2012}. In order to obtain a useful model, the inclusion of exogenous explanatory variables is desirable as well. We conjecture that the methods invoked in our proofs are also applicable in the context of those generalizations. Of course, the distribution under consideration must be good natured in the sense that, first, it allows for a proof of uniform mixing of the count process with geometrically decaying mixing coefficients, and, second, it is not too heavy tailed. The generalized Poisson distribution, for instance, seems to be suitable. This issue should be discussed in a separate paper.

\newpage

\begin{figure}[h!]
\centering
  \subfloat[3D plot of the true intensity function.]{\includegraphics[width = 0.45 \textwidth]{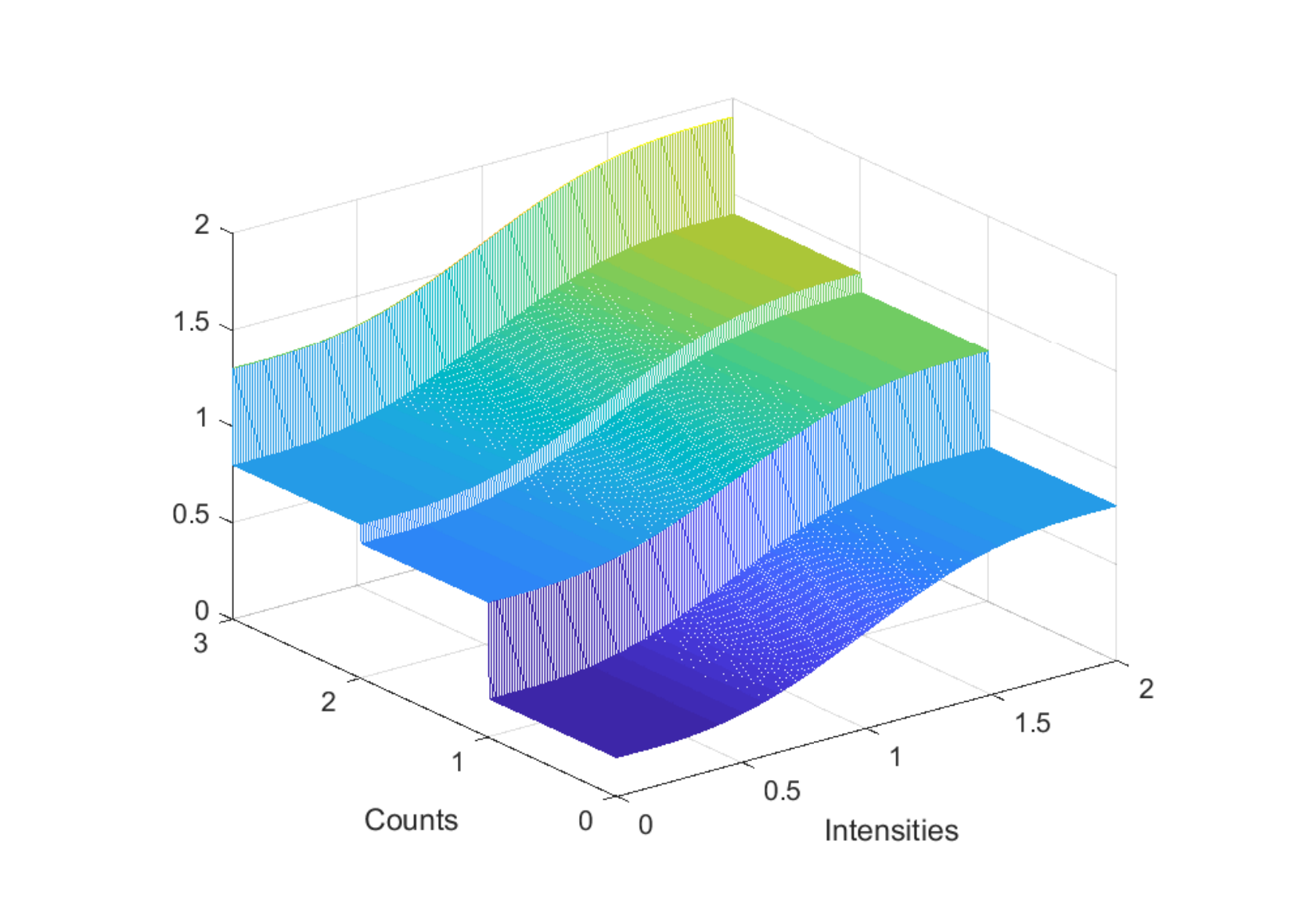}\label{fig:1.1}}
  \hspace{0.5cm}
  \subfloat[1000 realizations of the pairs $(\lambda_t,Y_t)$]{\includegraphics[width=0.45 \textwidth]{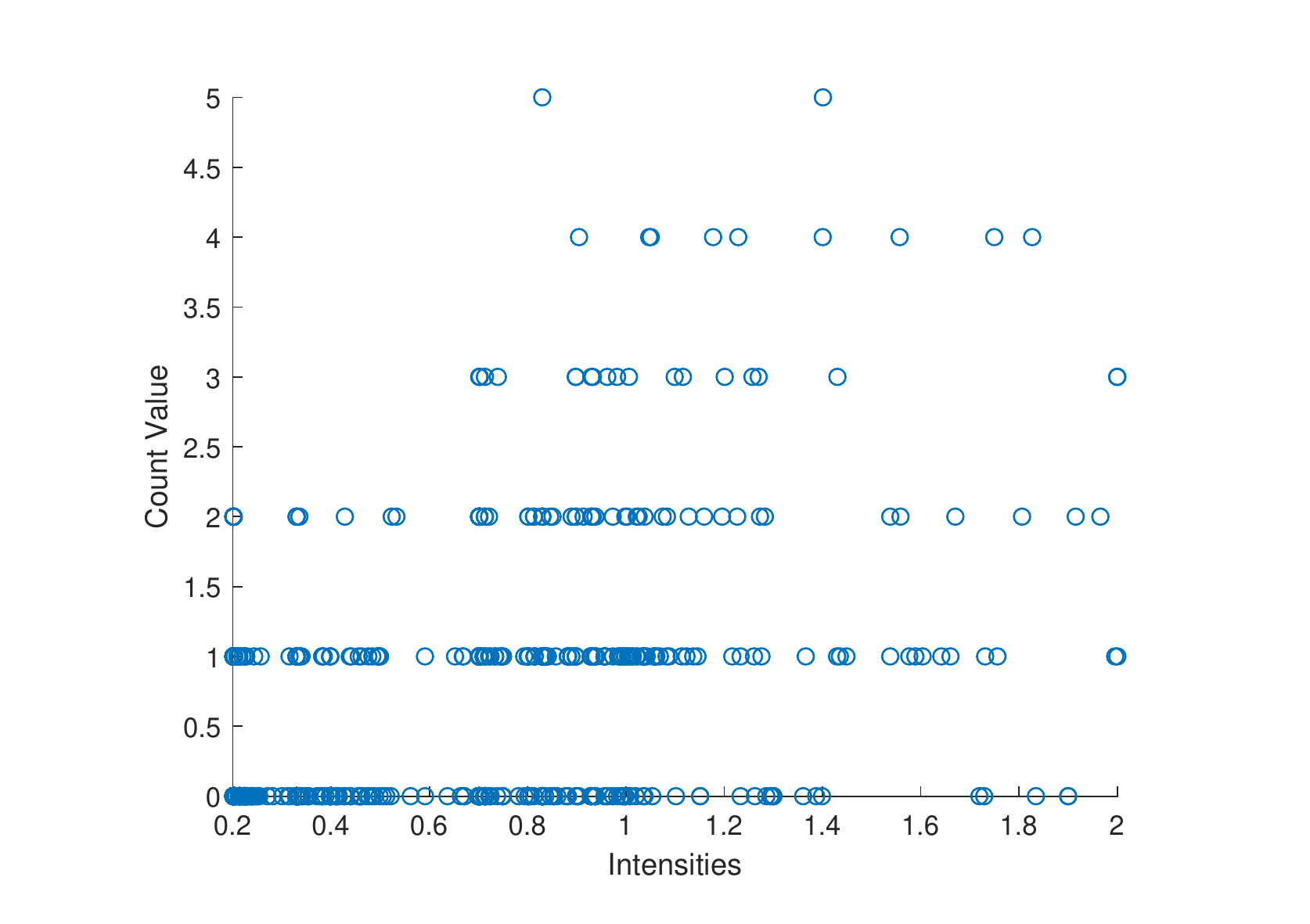}\label{fig:1.2}}
  \vspace{0.5cm}
   \caption{Underlying data for the numerical experiments. In (a) the true link function $m$ is shown; (b) shows the 1000 realizations of the count process that were used for estimating $m$.}\label{fig:1}
\end{figure}

\begin{figure}[h!]
\centering
  \subfloat[Estimation 1: MSE = 0.0383]{\includegraphics[width = 0.45\textwidth]{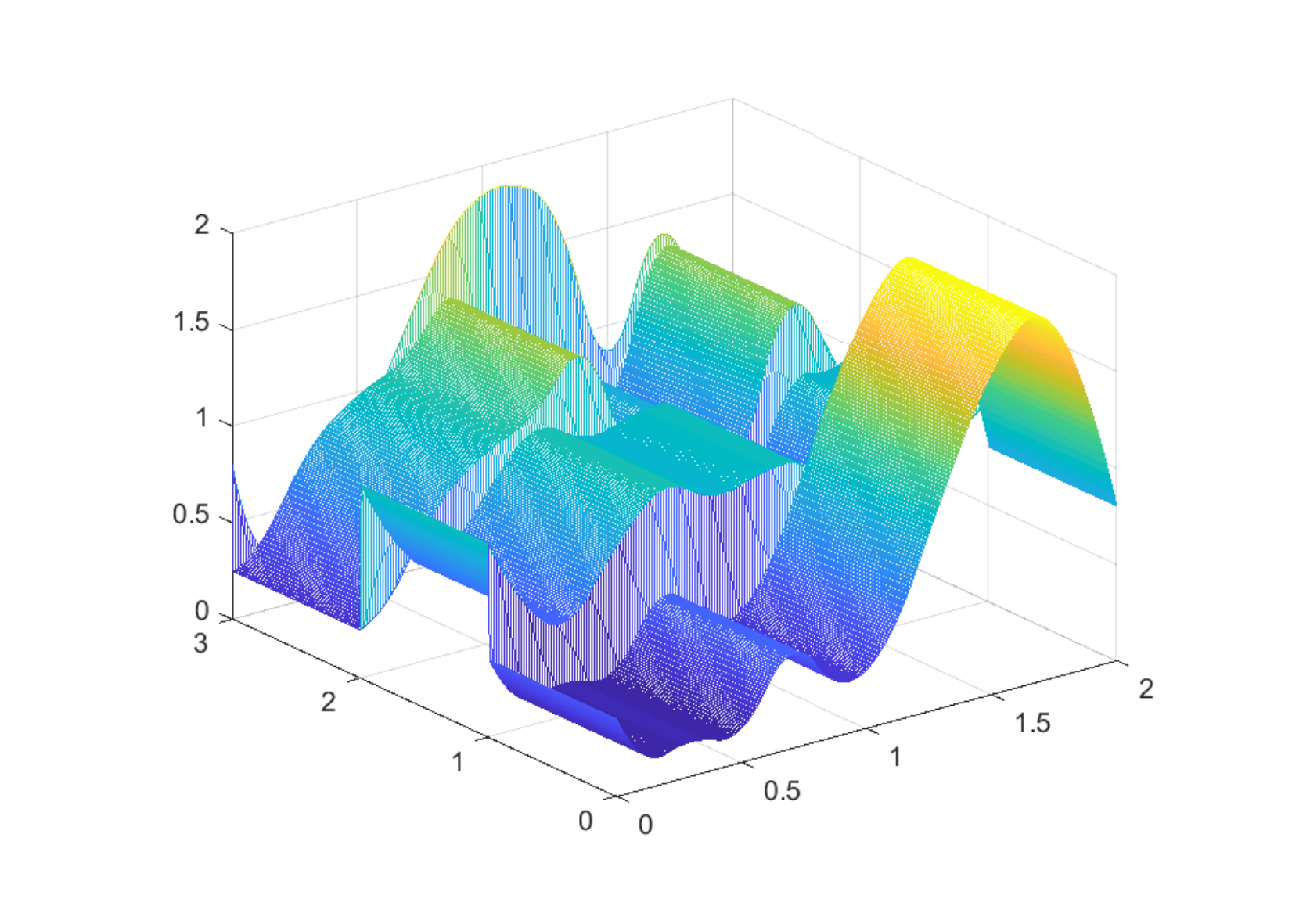}\label{fig:2.1}}
  \hspace{0.5cm}
  \subfloat[Estimation 2: MSE = 0.0277]{\includegraphics[width=0.45\textwidth]{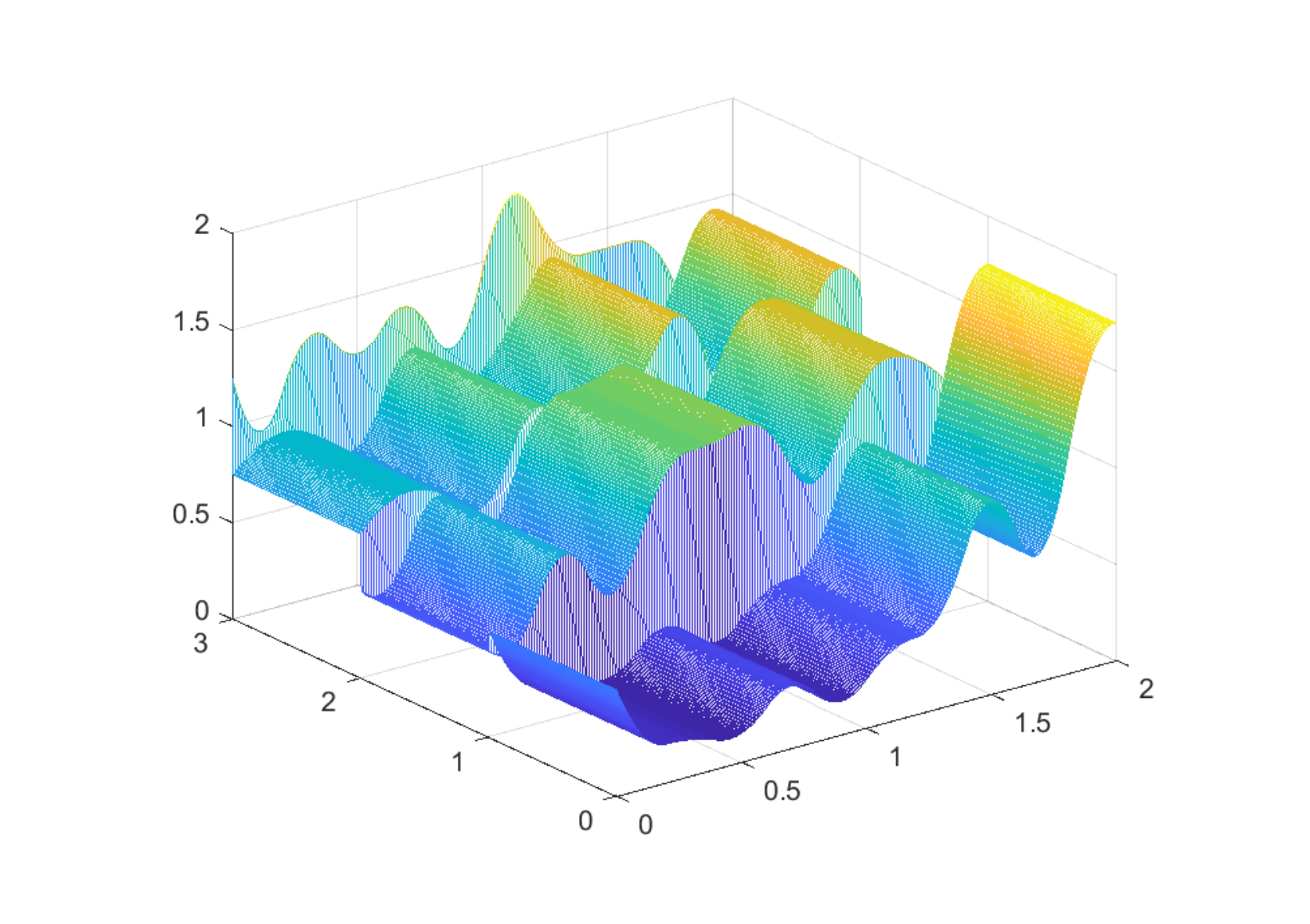}}
  \vspace{0.5cm}
  \subfloat[Estimation 3: MSE = 0.0278]{\includegraphics[width=0.45\textwidth]{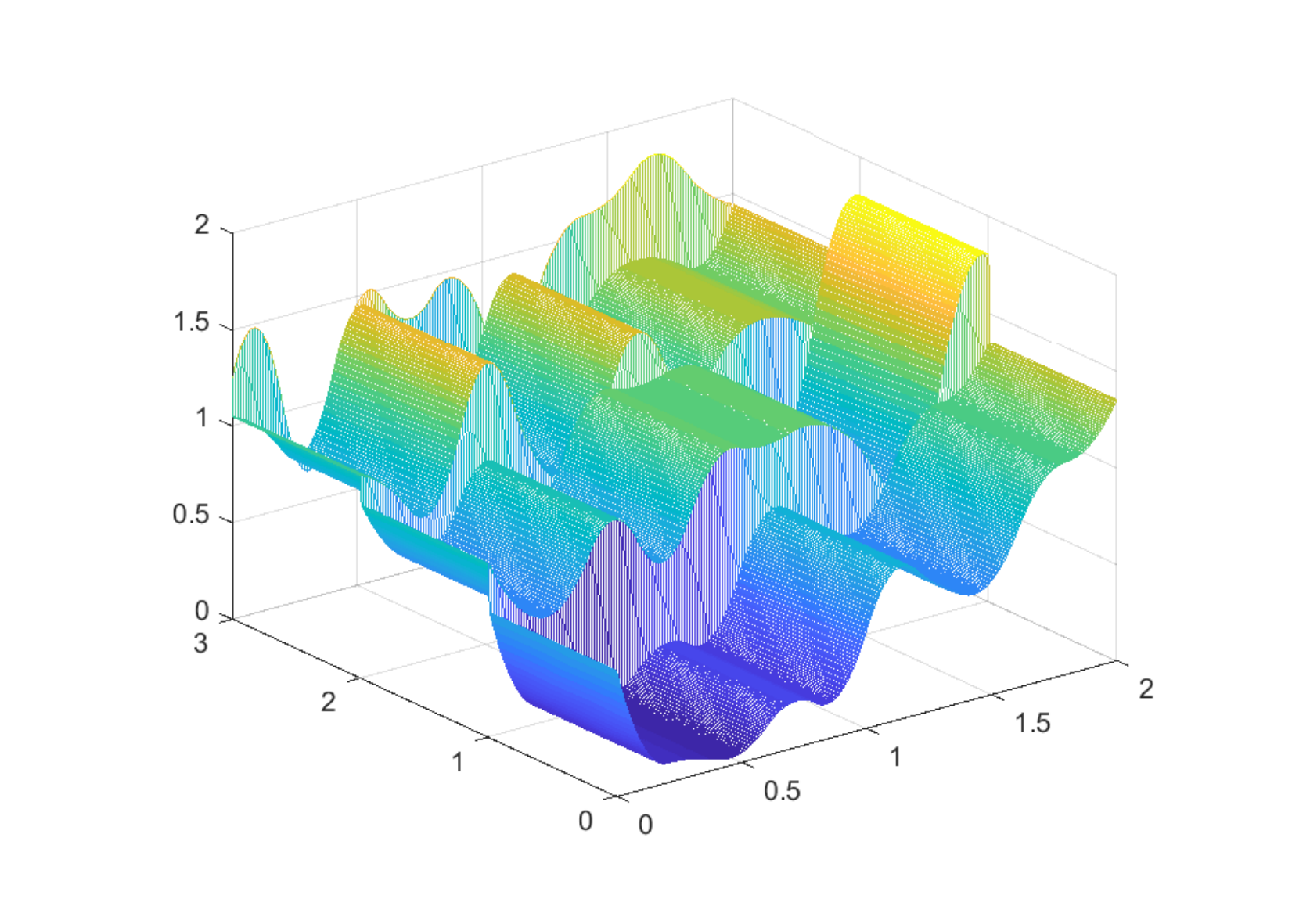}}
  \subfloat[Estimation 4: MSE = 0.0237]{\includegraphics[width=0.45\textwidth]{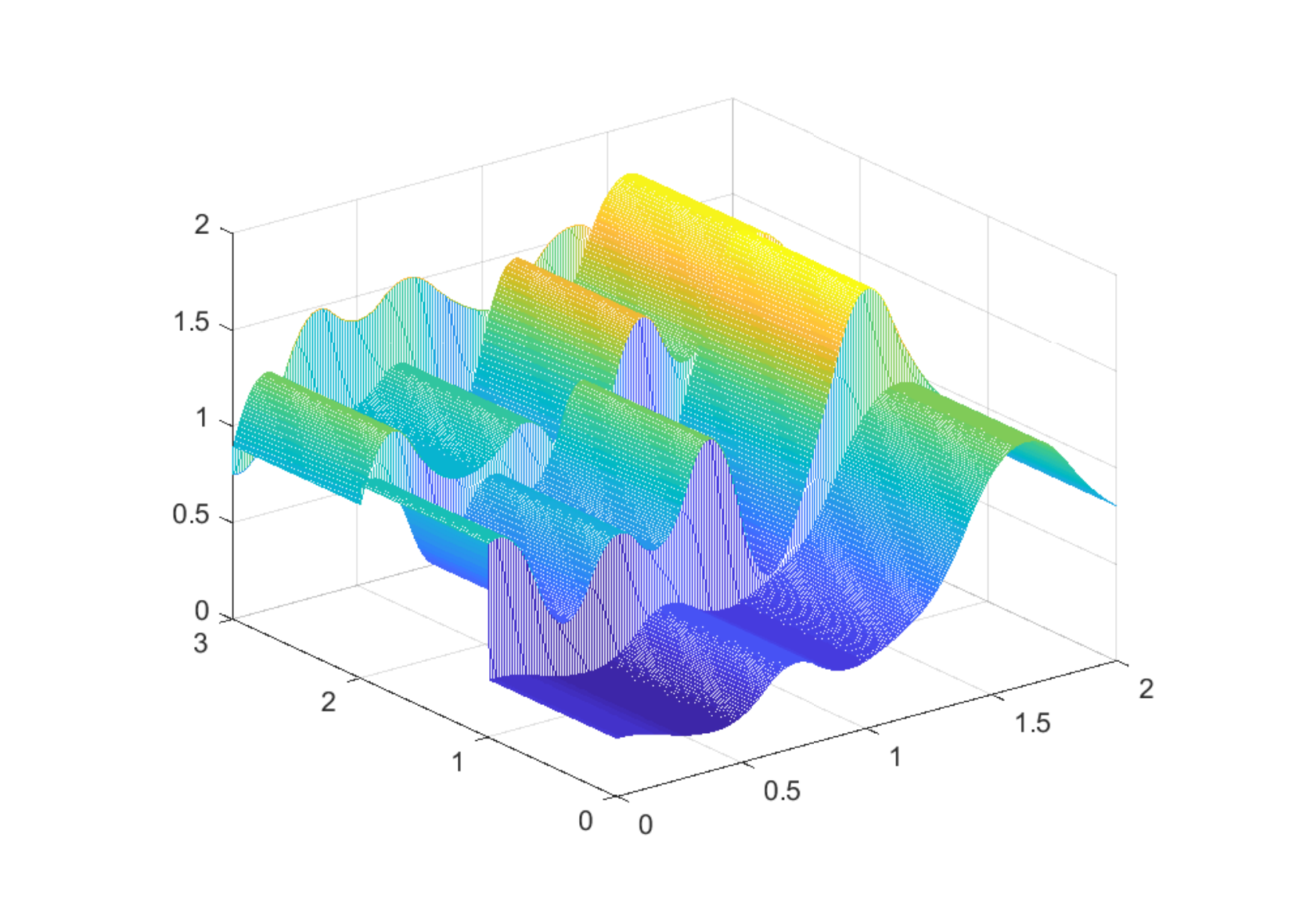}}
  \vspace{0.5cm}    
   \caption{Estimation with GA, $A_n = \{0,0.2,\ldots,1.8,2.0\}$, $n=1000$.}\label{fig:2}
\end{figure}

\newpage

\section{Proof of the main result}\label{sec:proof_main}

Before we proceed with the first auxiliary lemma, we fix the notation that we use throughout the rest of the proof.
\begin{defi}
\begin{enumerate}[label=(\roman*),wide,labelwidth=!,labelindent=0pt]
\item The true link function corresponding to the data generating process will be denoted by $m$, candidate functions in will typically be denoted by $g$ or $h$;
\item the conditional expectation of a random variable $X$ given that $\hat{m}_n = g$ is denoted by $\ew_{|\hat{m}_n = g}[X]$\,;
\item $\bm{Y}_k^l := (0,Y_k, \ldots, Y_l)$;
\item for any $g \in \cg$, $i \in \Zbb$ and $t \in \Nbb_+$, let $g^{[0]}(\lambda_i,Y_i) := g(\lambda_i,Y_i)$ and 
\[g^{[t]}(\lambda_{i-t},Y_{i-t},\ldots,Y_i) := g\big(g^{[t-1]}(\lambda_{i-t},Y_{i-t},\ldots,Y_{i-1}),Y_i\big)\,;\]
\item for any $g \in \cg$, $i \in \Zbb$, $t \in \Nbb$,
\begin{IEEEeqnarray*}{rRl}
f_t(g;\bm{Y}_{i-t}^{i+1}) 
&:=& 
\big( Y_{i+1} - m^{[t]}(\bm{Y}_{i-t}^i) \big)^2 - 
\big( Y_{i+1} - g^{[t]}(\bm{Y}_{i-t}^i) \big)^2 
 \\ 
&= &
\big( Y_{i+1} - m^{[t]}(0,Y_{i-t}, \ldots, Y_i) \big)^2 
 - \big( Y_{i+1} - g^{[t]}(0,Y_{i-t}, \ldots, Y_i) \big)^2 \,.
\end{IEEEeqnarray*}
\end{enumerate}
\end{defi}
\noindent
In the notation of part (ii), the loss can be written as a conditional expectation, $L(\hat{m}_n,m) = \ew_{|\hat{m}_n = g} \big[ m(\lambda'_0,Y'_0) - g(\lambda'_0,Y'_0)\big]^2$, where $(\lambda'_0,Y'_0) \sim \pi$ is independent of the original data generating process. 

For the proof of Theorem \ref{thm:main} we use some principles that are well known in the asymptotic analysis of least squares estimators for regression functions $r(x) := \ew_{|\xi_i = x}[\eta_i]$, $r \in \mathcal{R}$, with independent and identically distributed (i.i.d.) data pairs $(\eta_i,\xi_i) \sim P$ and some nonparametric class $\mathcal{R}$ of candidate functions. For the estimator $\hat{r}$ that minimizes the contrast functional $\varphi_n(g) = \frac{1}{n} \sum_{i=1}^n (\eta_i - g(\xi))^2$ over $\mathcal{R}$, the following central bound for the quadratic loss can be established:  
\begin{IEEEeqnarray}{rCl}\label{eq:risk_regression}
\int [r(x) - \hat{r}(x)]^2 \ P(dy,dx) 
& =   & \int [(y - \hat{r}(x))^2 - (y - r(x))^2] \ P(dy,dx) \\
&\leq & \big(\varphi_n(r) - \varphi_n(\hat{r})\big) - \int [(y - r(x))^2 - (y - \hat{r}(x))^2] \ P(dy,dx) \nonumber \\
&\leq & \sup_{g \in \mathcal{R}} \Big\{\big(\varphi_n(r) - \varphi_n(g)\big) - \ew \big(\varphi_n(r) - \varphi_n(g)\big) \Big\}\,.
\nonumber
\end{IEEEeqnarray}
A bound for the last term can be found by means of classical empirical process theory. 

We will adapt this line of argument to the nonparametric INGARCH(1,1) setting. However, due to the more complicated dependency structure of our data generating process, we will not be able to exploit as simple a relation as \eqref{eq:risk_regression}. As the univariate count process is easier to handle than the bivariate process, we seek as a starting point a relation similar to \eqref{eq:risk_regression} but adapted to the contrast functional $\Phi_n$ (cf. Definition \ref{defi:estimator}) and without the appearance of any $\lambda_i$. This is the subject of Lemma \ref{Lemma_1}. Subsequently, we discuss the dependency structure of the count process and prove that it is uniformly mixing, which leads to the coupling of Corollary \ref{lemma:2.3.3}. The resulting empirical process is then bounded in probability in Lemma \ref{lemma:chaining}, after which  the conclusion of the proof follows.

\begin{lemma}\label{Lemma_1}
Let $\delta = \delta(n) \asymp n^{-1/3} \log n$ 
and $t = t(n) = - \big\lceil\frac{2}{3 \log L_1} \log n \big\rceil$. There exists a positive constant $\gamma > 0$ such that for the set $\cg_k := \big\{ g \in \cg(B_n,M,L_1,L_2) \colon \ew[m^{[t]}(\bm{Y}_{0}^{t}) - g^{[t]}(\bm{Y}_{0}^{t})]^2 \leq 2^{2k+2} \gamma \delta^2 \big\}$, almost all $n \in \Nbb$, and $n \to \infty$
\begin{IEEEeqnarray*}{rCl}
\pr \Big\{ L(\hat{m}_n,m) > \delta^2 \Big\} & \leq  & \pr \bigcup_{k=0}^\infty \Bigg\{ \sup_{g \in \cg_k} 
 \frac{1}{n-t}\sum_{i=t}^{n-1}
 \big( f_t(g;\bm{Y}_{i-t}^{i+1}) -  \ew  f_t(g;\bm{Y}_{i-t}^{i+1}) \big)  > 2^{2k-2} \gamma \delta^2 \Bigg\} + o(1)\,.
\end{IEEEeqnarray*}
\end{lemma} 

Note that the stochastic process $\big\{\bm{Y}_{i-t}^{i+1}\big\}_{i \in \Zbb}$ is not an i.i.d. sequence. In order to bound the random functional $g \mapsto \frac{1}{n-t}\sum_{i=t}^{n-1}
 \big( f_t(g;\bm{Y}_{i-t}^{i+1}) -  \ew  f_t(g;\bm{Y}_{i-t}^{i+1}) \big)$ in probability uniformly over the set $\cg_k$, we want to use tools from the theory of empirical processes. These tools require the degree of dependence in the data generating process to be somewhat negligible. We will show that this is indeed the case because the count process is uniformly mixing.

\begin{defi} 
A stationary sequence $\{\eta_i\}_{i \in \Zbb}$ of $(\Rbb,\cb)$-valued random variables on $(\Omega,\cf,\pr)$
is called uniformly ($\phi-$)mixing if the sequence of mixing coefficients $\{\phi(k)\}_{k \in \Nbb}$\,,
\begin{IEEEeqnarray*}{rCl}
\phi(k) := \esssup \Big\{ \sup_{B \in \cb^\infty} \big| \pr\big\{ (\eta_k,\eta_{k+1}, \ldots) \in B \, \big| \, \eta_0, \eta_{-1}, \ldots \big\} - \pr\big\{ (\eta_k,\eta_{k+1}, \ldots) \in B \big\} \big| \Big\}\,,
\end{IEEEeqnarray*}
converges to zero as $k \to \infty$.
\end{defi}

It is well-known that uniform mixing implies absolute regularity ($\beta$-mixing), where the
corresponding coefficients are defined by
\begin{IEEEeqnarray*}{rCl}
\beta(k) := \ew \Big\{ \sup_{B \in \cb^\infty} \big| \pr\big\{ (\eta_k,\eta_{k+1}, \ldots) \in B \, \big| \, \eta_0, \eta_{-1}, \ldots \big\} - \pr\big\{ (\eta_k,\eta_{k+1}, \ldots) \in B \big\} \big| \Big\}
\end{IEEEeqnarray*}
\citep{Doukhan94}.
The most striking feature of a (uniformly) mixing sequence $\{\eta_i\}_{i \in \Zbb}$ is that $\eta_0$ and $\eta_k$ are almost independent if $k$ is large enough. Due to the stationarity of the process, the same is true for any pair $\eta_{i + k}$, $\eta_i$ with $i \in \Zbb$. The sense in which this almost independence materializes is specified by the next lemma which was used by \cite{Doukhan94} to prove exponential inequalities for absolute regular sequences. It is based on the classic coupling lemma by \cite{Berbee1979}. 

\begin{lemma}
Let $\{\eta_i\}_{ i \in \Zbb}$ be a uniformly mixing random sequence on $(\Omega,\cf,\pr)$, with mixing coefficients $\{\phi(k)\}_{k \in \Nbb}$. For $q \in \Nbb_+$, there exist two random sequences $\{\eta_i'\}_{i \in \Zbb}$ and $\{\eta_i^*\}_{i \in \Nbb}$ on a probability space $(S,\Sigma,P)$ such that 
\begin{enumerate}[label=(\roman*),wide,labelwidth=!,labelindent=0pt]
\item the processes $\{\eta_i\}$ and $\{\eta_i'\}$ have the same distribution;
\item the process $\{\eta_i^*\}$ is $q$-dependent, i.e. the block sequences $\big\{(\eta^*_{2jq},\ldots,\eta^*_{(2j+1)q-1})\colon j \in \Nbb\big\}$ and
$\big\{(\eta^*_{(2j+1)q},\ldots,\eta^*_{(2j+2)q-1})\colon j \in \Nbb\big\}$ are i.i.d., respectively;
\item $P^{\big(\eta'_{jq},\ldots,\eta'_{(j+1)q-1}\big)} = P^{\big(\eta^*_{jq},\ldots,\eta^*_{(j+1)q-1}\big)}$ for any $j \in \Nbb$\,;
\item $P\big\{\eta'_i \neq \eta^*_i, \text{ for some } i \in \{0,\ldots,n-1\}\big\} \leq \frac{n}{q}\phi(q)$ for any $n \in \Nbb_+$\,.
\end{enumerate}
\end{lemma}

The count process $\{Y_t\}_{t \in \Zbb}$ is uniformly mixing because the contraction property (C) renders neglectable the difference between the conditional distribution $\pr^{(Y_k,Y_{k+1}, \ldots) \,|\, Y_0}$ and the unconditional distribution $\pr^{(Y_k,Y_{k+1}, \ldots)}$ for large $k$. This property carries over to the process $\{\bm{Y}_{i-t}^{i+1}\}_{i \in \Zbb}$. We state the formal result without a proof, the interested reader can find a detailed proof in 
\cite{Wechsung2019}.
Mixing properties of a general class of GARCH-type processes including our data generating process are discussed by
\cite{Doukhan2017}.

\begin{lemma}\label{lemma:mixing_lagg}
The count process $\{Y_t\}_{t \in \Zbb}$ is uniformly mixing with mixing coefficients $\phi(k) \lesssim (L_1 + L_2)^k$. For any $t \in \Nbb$, the process $\{\bm{Y}_{i-t}^{i+1}\}_{n\in\Zbb}$ is uniformly mixing, and the corresponding mixing coefficients $\phi^{(t)}(k)$ are geometrically decreasing, $\phi^{(t)}(k) \lesssim (L_1+L_2)^{k-t}$.
\end{lemma}

\begin{folg}\label{lemma:2.3.3}
Let $t \in \Nbb$ be chosen as in Lemma \ref{Lemma_1}. For $q \in \Nbb_+$, there exist two sequences of $\Rbb^{t+3}$-valued random vectors, $V' = \{V'_i\}_{i \in \Zbb}$ and $V^* = \{V^*_{t+i}\}_{i \in \Nbb}$\,, on a probability space $(S, \Sigma, P)$ such that: $P^{V'} = \pr^{\{\bm{Y}_{i-t}^{i+1}\}}$, the sequence $V^*$ is $q$-dependent, $P^{(V'_{jq},\ldots,V'_{(j+1)q-1})} = P^{(V^*_{jq},\ldots,V^*_{(j+1)q-1})}$, and 
\begin{IEEEeqnarray*}{rCL}
P\big\{ V_i' \neq V_i^* \text{ for some } i \in \{t, \ldots, n-1\} \big\} \leq \frac{n-t}{q} \phi^{t}(q)\,.
\end{IEEEeqnarray*}
\end{folg}
Note that for any $i$, the first component of $V_i^*$ can be non-zero only on a $P$-null set. Since such a set is neglectable for the further argument, we stipulate without loss of generality that $V_i^* = (0,Y^*_{i-t}, \ldots, Y^*_{i+1})$ with $P^{Y^*_i} = \pr^{Y_i}$. At this point, we have to introduce some further notation.

\begin{defi}
\begin{enumerate}[label=(\roman*),wide,labelwidth=!,labelindent=0pt]
\item For any $i \in \Zbb$, the random variable $Y_i^*$ on $(S,\Sigma,P)$ is defined as the last coordinate of $V_{i-1}^*$; the random vector $Z_i^*$ is defined as $Z_i^* := (0,Y_{i-t}^*, \ldots, Y_i^*)$.
\item The expectation with respect to $P$ is denoted by $E$.
\item $\lambda_{i}^* := E[Y_i^* \, | \,Y_{i-1}^*,Y^*_{i-2},\ldots ]$\,.
\item For functions $g\in \cg$ and some length $q$, we introduce
\begin{IEEEeqnarray}{rCl}
N 
 := 
\begin{cases}
\frac{1}{2} \left \lfloor \frac{n-t}{q} \right \rfloor \qquad &\text{ if } \left \lfloor \frac{n-t}{q} \right \rfloor \text{ is even,} \\
 \frac{1}{2} \left( \left \lfloor \frac{n-t}{q} \right \rfloor - 1 \right)  \qquad &\text{ if } \left \lfloor \frac{n-t}{q} \right \rfloor \text{ is odd.}
\end{cases} \\ 
X^*_{r}(g) 
 :=  \frac{1}{q} \sum_{i=0}^{q-1} \big( f_t(g;V^*_{t+rq+i})- E \, f_t(g;V^*_{t+rq+i}) \big) 
 \\ 
R_n(g) 
 :=  \frac{1}{n-t} \sum_{i=2Nq}^{n-1-t} \big( f_t(g;V_{t+i}^*) - E f_t(g;V_{t+i}^*) \big)\,.
\end{IEEEeqnarray}
\end{enumerate}
\end{defi}

\begin{rem}\label{Remark1}
The coupling lemma implies that $X^*_{r}(g)$ and $X^*_{r+2}(g)$ are independent for any $r \in \Nbb$ and $g \in \cg$. Using the equality 
\begin{IEEEeqnarray}{rCl}
\frac{1}{n-t} \sum_{i=t}^{n-1}\big( f_t(g;V^*_{i}) - E f_t(g;V_{i}^*) \big)  
& = & \frac{q}{n-t} \sum_{j=0}^{N-1}  X^*_{2j}(g)  + \frac{q}{n-t} \sum_{j=0}^{N-1}  X^*_{2j+1}(g)  + R_n(g)
\label{eq:blocks}
\end{IEEEeqnarray}
and the triangle inequality for probabilities, we obtain on the basis of Lemma \ref{Lemma_1} (ii) and Corollary \ref{lemma:2.3.3}, for almost all $n \in \Nbb$
\begin{IEEEeqnarray}{rCl}
 \pr \Big\{ L(\hat{m}_n,m) > \delta^2  \Big\}  
&  \leq & \sum_{k=0}^\infty 2 \ P \Bigg\{ \sup_{g \in \cg_k} \frac{q}{n-t} \sum_{j=0}^{N-1}  X^*_{2j}(g) > 2^{2k-4} \gamma \delta^2 \Bigg\} \label{eq:empirical_1} \\
&    & + \sum_{k=0}^\infty  P \Bigg\{ \sup_{g \in \cg_k} R_n(g) > 2^{2k-3} \gamma \delta^2 \Bigg\}  + \frac{n-t}{q} \phi^t(q) + o(1)
\label{eq:Rn} 
\end{IEEEeqnarray} 
as $n \to \infty$. 
\end{rem}

The functional in line \eqref{eq:empirical_1}, $g^{[t]} \mapsto \frac{q}{n-t}  \sum_{j=0}^{N-1} X^*_{2j}(g)$, can be viewed as the trajectory of an empirical process driven by i.i.d. random variables, indexed by the function class $\{g^{[t]}\colon g \in \cg_k\}$. A standard tool to find uniform bounds for these trajectories is the so called chaining technique 
\citep{vandegeer1990, vandervaart1996, gyorfi2002, gine2015}. It is based on uniform bounds for the trajectories over successively refined finite approximations of the index set. Therefore, estimating the size of the index set in terms of covering numbers is of central importance. Since the functions $g^{[t]}$ have a high dimensional domain, which induces large covering numbers for the respective function class, it is clearly desirable to have a link to the considerably less complex class $\cg_k$. In this context, the usual strategy, considering covering numbers with respect to the $L_2$ norm, is infeasible since it is unclear to us how the $L_2(P^{(\lambda_i^*, Y_i^*)})$ norm of a function $g \in \cg$ and the $L_2(P^{Z_i^*})$ norm of the $t$-fold iterated function $g^{[t]}$ are related. However, quite easily we can find a relation between the respective $\|\cdot\|_\infty$-norms. This is done in Lemma \ref{lemma:covering} \ref{lemma:f_g}, which will also enable us to bound the remainder term in line \eqref{eq:Rn}. Part \ref{prop:g_s}, in turn, provides a bound for the $\|\cdot\|_\infty$-covering number of the index set $\cg_k$.

\begin{lemma}\label{lemma:covering}
\begin{enumerate}[label=(\roman*),wide,labelwidth=!,labelindent=0pt]
\item \label{lemma:f_g} For any $g,h \in \cg(M,L_1,L_2)$, 
\begin{IEEEeqnarray*}{rCL}
\big| f_t(g;V_i^*) - f_t(h;V_i^*) \big| \leq 
\frac{2 (Y_{i+1}^* + M)}{1-L_1} \, \|g-h\|_\infty\,. 
\end{IEEEeqnarray*} 
\item \label{prop:g_s}
For any $k,s \in \Nbb$, there exists a set $\cg_k^{(s)} \subset \cg_k$ with at most $e^{2MB_n2^{s-k}/(\sqrt{\gamma}\delta)}$ elements and a selection function $\pi_{s,k} \colon \cg_k \to \cg_k^{(s)}$ such that
$\|g - \pi_{s,k}g\|_\infty \leq 2^{-s} 2^{k+1} \sqrt{\gamma} \delta$ for any $g \in \cg_k$. We stipulate the notation $\pi_{s,k} g =: g_{s,k}$.
\end{enumerate}
\end{lemma}   

\begin{rem}\label{Remark2}
As a simple consequence of Lemma \ref{lemma:covering} \ref{lemma:f_g}, we obtain a bound for the remainder term $R_n$,
\begin{IEEEeqnarray}{rCl}
P \Bigg\{ \sup_{g \in \cg_k} R_n(g) > 2^{2k-3} \gamma \delta^2 \Bigg\} 
\lesssim 2^{-2k} \delta^{-2} \frac{q}{n-t}\,. \label{eq:Rn_bound}
\end{IEEEeqnarray} 
This follows from Markov's inequality in combination with the inequality
\begin{IEEEeqnarray}{+rCl+x*}
E \big[ \sup_{g \in \cg} |R_n(g)| \big] 
& \leq & E \Big[ \sup_{g \in \cg} \frac{1}{n-t} \sum_{i=2Nq}^{n-1} \big| f_t(g;V^*_{t+i})\big| + E \big|f_t(g;V^*_{t+i}) \big| \Big] \nonumber \\
& \leq & \frac{1}{n-t} \sum_{i=2Nq}^{n-1} \frac{2 M}{1-L_1} \ 2\, E \big( Y_{t+i+1}^* + M \big) \nonumber \\
& \leq & \frac{q}{n-t} \frac{16\,M^2}{1-L_1}\,,
\end{IEEEeqnarray}
which follows from Lemma \ref{lemma:covering} \ref{lemma:f_g} and the fact that $f_t(m;V_i^*) = 0$.
\end{rem}

Bounding the empirical process in line \eqref{eq:empirical_1} by means of the chaining argument is much more technical. As often in these situations, it is convenient to add some extra randomness in terms of an independent sequence of Rademacher variables. This technique, which is formalized in the following lemma, is well established and has a wide range of applications, e.g. in proving uniform laws of large numbers 
\citep{gine2015,gyorfi2002}. 

\begin{lemma}\label{symmetrisierungslemma} 
On a probability space $(E,\ce,P)$\,, let $Y_1,\ldots,Y_n$ be $\Rbb^d$-valued random variables and $\varepsilon_1,\ldots,\varepsilon_n$ independent Rademacher variables. Let $\ch$ be a class of continuous functions $h\colon \Rbb^d \to \Rbb$ such that $\int |h(Y)|^2 \, dP < \infty$. If $\sup_{h \in \ch} \ \var \Big( \sum_{i=1}^n h(Y_i) \Big) \leq \frac{t^2}{8}$,
\begin{IEEEeqnarray*}{rCl}
P \Bigg\{ \sup_{h \in \ch} \big| \sum_{i=1}^n h(Y_i) \big| > t \Bigg\}
\leq 4 \, P \Bigg\{ \sup_{h \in \ch} \big| \sum_{i=1}^n \varepsilon_i \, h(Y_i) \big| > t/4 \Bigg\}\,.
\end{IEEEeqnarray*}
\end{lemma}
In Lemma \ref{lemma:variance}, Part (ii) applies Lemma \ref{symmetrisierungslemma} to the sums $\frac{q}{n-t} \sum_{j=0}^{N-1}  X^*_{2j}(g)$, for which Part (i) supplies the requisite uniform variance bound.  

\begin{lemma}\label{lemma:variance}
Suppose that $\delta(n) \asymp n^{-1/3} \log n$ and $t(n) = - \big\lceil \frac{2}{3 \log L_1} \log n \big\rceil$. Let the length of the blocks $X_{2j}^*(g)$ also depend on the sample size in such a way that $q(n) \asymp t(n)$. Recall the constant $\gamma$ introduced in Lemma \ref{Lemma_1} (ii).
Then
\begin{enumerate}[label=(\roman*),wide,labelwidth=!,labelindent=0pt]
\item \label{i} $\sup_{g \in \cg_k} \,  \var \left( \frac{q}{n-t} \sum_{j=0}^{N-1} X^*_{2j}(g) \right) \leq (M^2+2\,M^{3/2}+M) \, 2^{2k+4} \gamma \delta^2 \frac{q}{n-t} $\,,
\item for all $k \in \Nbb$ and almost all $n \in \Nbb$
\begin{IEEEeqnarray*}{rCl}
 P \Bigg\{ \sup_{g \in \cg_k} 
\frac{q}{n-t} \sum_{j=0}^{N-1} X^*_{2j}(g) > 
2^{2k-4} \gamma \delta^2 \Bigg\}  \nonumber
\leq 4 P \Bigg\{ \sup_{g \in \cg_k}
 \frac{q}{n-t}  \sum_{j=0}^{N-1} \varepsilon_j  X^*_{2j}(g) > 2^{2k-6} \gamma \delta^2 \Bigg\}\,. \label{eq:empirical}
\end{IEEEeqnarray*}
\end{enumerate}
\end{lemma}

We have finally obtained a form of the empirical process that is well suited for the application of the chaining argument. The resulting bound is presented in the following lemma. 
  
\begin{lemma} \label{lemma:chaining}
There exists a constant $C_\delta > 0$ and a natural number $n_0$ such that for $\delta = C_\delta \ n^{-1/3} \log n$ 
\begin{IEEEeqnarray*}{rCl}
 P \Bigg\{ \sup_{g \in \cg_k} \frac{q}{n-t} 
&  & \sum_{j=0}^{N-1} \varepsilon_j  X^*_{2j}(g) > 2^{2k-6} \gamma \delta^2 \Bigg\} \\
&  & \hspace{1.7cm} \leq C \, 2^{-2k} \frac{\log n}{n} + e^M \, n^{-(k+1)} + \exp\Big( - C' \, n^{1/3} \ 2^k \Big) + C'' \ 2^{-k} (\log n)^{-1/2}
\end{IEEEeqnarray*}
for all $k \in \Nbb$, $n > n_0$, and some positive constants $C,C',C''$.
\end{lemma}

Combining the estimate in Remark \ref{Remark1} with Lemma \ref{lemma:chaining} and Remark \ref{Remark2}, we infer that there exists a constant $C_\delta > 0$ such that for all but finitely many $n$ and all $k \in \Nbb$
\begin{IEEEeqnarray}{rCl}
 \pr \Big\{ L(\hat{m}_n,m) > C_\delta^2 \ n^{-2/3} (\log n)^2  \Big\} 
& \leq & \tilde{C} \big( b_n + \sum_{k=0}^\infty a_{n,k} \big) \,,
\end{IEEEeqnarray}
with
\begin{IEEEeqnarray}{rCl}
a_{n,k} 
& := & 2^{-2k} 
\frac{\log n}{n} + n^{-k}  n^{-1} + \exp\Big( - C' \, n^{1/3} \ 2^k \Big) +  2^{-k} (\log n)^{-1/2}  + 2^{-2k} \delta^{-2} \frac{q}{n-t}\,,\\
b_n  
& := & \frac{n-t}{q} \phi^t(q) + o(1)\,,
\end{IEEEeqnarray}
and some constant $\tilde{C} > 0$. Recall that $t(n) = - \big \lceil \frac{2}{3 \log L_1} \log n \big \rceil$ and $q(n) \asymp t(n)$. These facts imply
$\delta_n^{-2} \frac{q}{n-t} 
 \asymp  n^{-1/3} (\log n)^{-1}$, 
whence we infer that $\lim_{n \to \infty} a_{n,k} = 0$ for all $k\in \Nbb$ . Since there exists an absolute summable sequence 
$\{\eta_k\}_{k \in \Nbb} \subset \Rbb$ such that $\sup_{n \geq 2} |a_{n,k}| \leq \eta_k$ for any $k$, we conclude that $\sum_{k=0}^\infty a_{n,k} \to 0$ as $n \to \infty$. Because $\phi^t(q) \lesssim (L_1 + L_2)^{q-t}$, the sequence $\{q(n)\}$ can be specified such that
$\frac{n-t}{q} \phi^t(q) \asymp  (\log n)^{-1}$,
implying that
$\lim_{n \to \infty} b_n = 0$ as well. The proof of Theorem \ref{thm:main} is now complete.

\section{Proofs of auxiliary results}\label{sec:proof_aux}

\begin{proof}[Proof of Lemma \ref{Lemma_1}]
First of all, note that $\limsup_{n \to \infty} L_1^t / \delta^2 = 0$. Hence, for some $\varepsilon > 0$ there exists a number $n_0$ such that $\varepsilon \big( L(\hat{m}_m,m) \big)^{1/2} > M \, L_1^t$ for all $n>n_0$ and all $\omega \in \Omega_0 = \big\{ \omega \in \Omega \colon L(\hat{m}_n,m) > \delta^2 \big\}$. After some straightforward calculations exploiting the assumption that both $m$ and $\hat{m}_n$ possess the contraction property, we can show that for $n > n_0$
\begin{IEEEeqnarray*}{rCl}
&     & \ew_{|\hat{m}_n = g} \Big[ \underbrace{\left( Y'_{t+1} - g^{[t]}(0,Y'_0, \ldots, Y'_{t}) \right)^2 
                          - \left( Y'_{t+1} - m^{[t]}(0,Y'_0, \ldots, Y'_{t}) \right)^2}_{= -f_t(g;\bm{Y'}\,_0^{t+1})} \Big]\nonumber\\
&     & \hspace{3cm} \geq \ew_{|\hat{m}_n = g} \left[m^{[t]}(0,Y'_0, \ldots, Y'_{t})	- g^{[t]}(0,Y'_0, \ldots, Y'_{t}) \right]^2                          - 3 M^2  L_1^{t}                              
\end{IEEEeqnarray*}
and 
\begin{IEEEeqnarray}{rCl}
 \ew_{|\hat{m}_n = g} \left[m^{[t]}(0,Y'_0, \ldots, Y'_{t})	- g^{[t]}(0,Y'_0, \ldots, Y'_{t}) \right]^2  \nonumber
&    &  \geq  \frac{(1-\varepsilon)^2}{12}\, L(\hat{m}_n,m) - 2 \, M^2 L_1^{t}\,,
\end{IEEEeqnarray}
for almost all $\omega \in \Omega_0$ \citep[pp. 42--47]{Wechsung2019}. We define the constant $\gamma := (1-\varepsilon)^2/24$ and conclude that there exists a number $n_1 \geq n_0$ such that for almost all $\omega \in \Omega_0$ and all $n > n_1$,
     \begin{IEEEeqnarray}{rCl}
    &   & \ew_{|\hat{m}_n = g} \left[ m^{[t]} (0, Y'_0, \ldots, Y'_{t}) - g^{[t]}(0, Y'_0, \ldots, Y'_{t}) \right]^2  
    >   \frac{(1-\varepsilon)^2}{12} \delta^2 - 2\,M^2 L_1^t
   \geq  \gamma \delta^2\,, \label{eq:cond1} \\
    &    & \ew_{|\hat{m}_n = g} \big[ -f_t(g;\bm{Y'}\,_0^{t+1}) \big]  
      > \ew_{|\hat{m}_n = g} \left[ m^{[t]} (0, Y'_0, \ldots, Y'_{t}) - g^{[t]}(0, Y'_0, \ldots, Y'_{t}) \right]^2 - \frac{\gamma}{2} \delta^2\,.  \label{eq:cond2}
\end{IEEEeqnarray}
In other words, $\Omega_0 \subset \big\{\omega \in \Omega \colon \text{\eqref{eq:cond1} and \eqref{eq:cond2} hold}\big\}$ for $n > n_1$, up to a null set. 

For any $g \in \cg_n^{\textit{cand}}$, we have $\Phi_n(g) - \Phi_n(\hat{m}_n) \geq 0$. Now let the random function $m_n \in \cg_n^{\textit{cand}}$ be defined by $m_n(\lambda,y) = m(\lambda,y \wedge Y_{(n)})$ for all $\lambda \leq M$ and $y \in \Nbb$. Then $\Phi_n(m) = \Phi_n(m_n)$ and we conclude that
$\frac{1}{n} \sum_{i=0}^{n-1} f_i(\hat{m}_n;\bm{Y}\,_{0}^{i+1}) = \Phi_n(m) - \Phi_n(\hat{m}_n) \geq 0$. Furthermore, by equation \eqref{eq:cand}, $\pr\big\{\hat{m}_n \notin \cg(B_n,M,L_1,L_2)\big\} = o(1)$. Now let $Y'_0, \ldots, Y'_n$ be a ghost sample with the same distribution as $Y_0,\ldots,Y_n$ but independent of the data generating process, and define $\big(0,Y'_0, \ldots, Y'_t \big) =: \bm{Y'}\,_0^{t}$. Then
\begin{IEEEeqnarray}{rCl}
&      & \pr \Big\{  L(\hat{m}_n,m) > \delta^2 \ \Big\}\nonumber \\  
& \leq & \pr  \Bigg\{ \ew_{|\hat{m}_n = g} \left[ m^{[t]}(\bm{Y'}\,_{0}^{t}) - g^{[t]}(\bm{Y'}\,_{0}^{t}) \right]^2 > \gamma \delta^2 \, ; \nonumber \\
&      & \hspace{1cm}        
\ew_{|\hat{m}_n = g} \left[ - f_{t}(g;\bm{Y'}\,_{0}^{t+1}) \right] > \ew_{|\hat{m}_n = g} \left[ m^{[t]}(\bm{Y'}\,_{0}^{t}) - g^{[t]}(\bm{Y'}\,_{0}^{t}) \right]^2 -  \frac{\gamma}{2} \delta^2 \Bigg\} \nonumber \\
& \leq &  \pr  \Bigg\{ \ew_{|\hat{m}_n = g} \left[ m^{[t]}(\bm{Y'}\,_{0}^{t}) - g^{[t]}(\bm{Y'}\,_{0}^{t}) \right]^2 > \gamma \delta^2 \, ; \nonumber \\
&      & \hspace{1cm} \frac{1}{n}\sum_{i=0}^{n-1} f_i(\hat{m}_n;\bm{Y}\,_{0}^{i+1}) -  \ew_{|\hat{m}_n = g} \left[ f_{t}(g;\bm{Y'}\,_{0}^{t+1}) \right] \nonumber \\
&      & \hspace{4cm} > \ew_{|\hat{m}_n = g} \left[ m^{[t]}(\bm{Y'}\,_{0}^{t}) - g^{[t]}(\bm{Y'}\,_{0}^{t}) \right]^2 -  \frac{\gamma}{2} \delta^2 \Bigg\} \nonumber \\
&\leq  &  \pr \Bigg\{ \exists \, g \in \cg(B_n,M,L_1,L_2) \colon   \ew \left[ m^{[t]}(\bm{Y}_{0}^{t}) - g^{[t]}(\bm{Y}_{0}^{t}) \right]^2 > \gamma \delta^2 \,; \label{eq:2.12}          \\
&      & \hspace{1cm}  \frac{1}{n}\sum_{i=0}^{n-1} f_i(g;\bm{Y}_{0}^{i+1}) -  \ew  f_t(g;\bm{Y}_{0}^{t+1})  > \ew \left[ m^{[t]}(\bm{Y}_{0}^{t}) - g^{[t]}(\bm{Y}_{0}^{t}) \right]^2 - \frac{\gamma}{2} \delta^2 \Bigg\} + o(1) \nonumber
\end{IEEEeqnarray}
for $n > n_1$. Using the decomposition
\begin{IEEEeqnarray*}{rCl}
\left\{ \ew \left[ m^{[t]}(\bm{Y}_{0}^{t}) - g^{[t]}(\bm{Y}_{0}^{t}) \right]^2 > \delta^2 \gamma  \right\} 
&  = & \bigcup_{k=0}^\infty \left\{ 2^{2k+2}   \delta^2 \gamma  
\geq \ew \left[ m^{[t]}(\bm{Y}_{0}^{t}) - g^{[t]}(\bm{Y}_{0}^{t}) \right]^2 >  2^{2k}  \delta^2 \gamma  \right\} \hfill
\end{IEEEeqnarray*}
and introducing $\cg_k :=\big\{g \in \cg(B_n,M,L_1,L_2) \colon \ew[m^{[t]}(\bm{Y}_{0}^{t}) - g^{[t]}(\bm{Y}_{0}^{t})]^2 \leq 2^{2k+2} \gamma \delta^2\big\}$, we write \eqref{eq:2.12} further as
\begin{IEEEeqnarray}{rCl}
&     &\pr \bigcup_{k=0}^\infty                   \Bigg\{ \exists \, g \in \cg(B_n,M,L_1,L_2): \  2^{2k+2} \gamma \delta^2 \geq  \ew \left[ m^{[t]}(\bm{Y}_{0}^{t}) - g^{[t]}(\bm{Y}_{0}^{t}) \right]^2 > 2^{2k} \gamma \delta^2  \,; 
\nonumber \\
&     &                                              \qquad \qquad \frac{1}{n}\sum_{i=0}^{n-1}
 f_i(g;\bm{Y}_{0}^{i+1}) -  \ew  f_t(g;\bm{Y}_{0}^{t+1})  >  \ew \left[ m^{[t]}(\bm{Y}_{0}^{t}) - g^{[t]}(\bm{Y}_{0}^{t}) \right]^2  - \frac{\gamma}{2} \delta^2 \Bigg\} \nonumber \\
&\leq &  \pr \bigcup_{k=0}^\infty \Bigg\{ \sup_{g \in \cg_k} 
 \frac{1}{n}\sum_{i=0}^{n-1}
 f_i(g;\bm{Y}_{0}^{i+1}) -  \ew  f_t(g;\bm{Y}_{0}^{t+1})  > 2^{2k-1} \gamma \delta^2 \Bigg\}\,.   
\end{IEEEeqnarray}

This is almost the statement of the lemma. We just have to substitute $\frac{1}{n}\sum_{i=0}^{n-1}
 f_i(g;\bm{Y}_{0}^{i+1})$ with $\frac{1}{n-t} \sum_{i=t}^{n-1} f_t(g;\bm{Y}_{i-t}^{i+1})$. To that end, we invoke the triangle inequality for probabilities and the fact that by stationarity $\ew  f_t(g;\bm{Y}_{0}^{t+1}) = \ew  f_t(g;\bm{Y}_{i-t}^{i+1})$ for $i \geq t$ and conclude,
\begin{IEEEeqnarray}{rCl}
&           & \pr \bigcup_{k=0}^\infty \Bigg\{ \sup_{g \in \cg_k} 
 \frac{1}{n}\sum_{i=0}^{n-1}
 f_i(g;\bm{Y}_{0}^{i+1}) -  \ew  f_t(g;\bm{Y}_{0}^{t+1})  > 2^{2k-1} \gamma \delta^2 \Bigg\} \nonumber  \\
 & \leq     & \pr \bigcup_{k=0}^\infty \Bigg\{ \sup_{g \in \cg_k} 
 \frac{1}{n-t}\sum_{i=t}^{n-1}
 \big( f_t(g;\bm{Y}_{i-t}^{i+1}) -  \ew  f_t(g;\bm{Y}_{i-t}^{i+1}) \big)  > 2^{2k-2} \gamma \delta^2 \Bigg\} \nonumber \\
 &          & \qquad \qquad + \pr \Bigg\{\underbrace{ \sup_{g \in \cg} 
  \Big| \frac{1}{n} \sum_{i=0}^{n-1} f_i(g;\bm{Y}_0^{i+1}) - \frac{1}{n-t} \sum_{i=t}^{n-1} f_t(g;\bm{Y}_{i-t}^{i+1}) \Big|}_{=: \Delta_n} > \frac{ \gamma \delta^2 }{4} \Bigg\}\,.  \label{eq:Delta_n}
\end{IEEEeqnarray}

It can be shown that $\ew \Delta_n \lesssim \frac{t}{n} + L_1^t$. This is again done by straightforward computations using the contraction property and the uniform boundedness of the candidate functions. Essentially, the term $\frac{t}{n}$ is owed to the different lengths of the sums ($n$ and $n-t$ addends, respectively), and the term $L_1^t$ is the order of the difference $\big| f_i(g;\bm{Y}_0^{i+1}) - f_t(g;\bm{Y}_{i-t}^{i+1}) \big|$. The latter fact is a consequence of the contraction property and the fact that all functions in $\cg$ are bounded by $M$. A simple application of Markov's inequality shows then that the probability in line \eqref{eq:Delta_n} converges to zero as
$n \to \infty$. \qedhere
\end{proof}

\begin{proof}[Proof of Lemma \ref{lemma:covering}] 
\begin{enumerate}[label=(\roman*),wide,labelwidth=!,labelindent=0pt]
\item Recall that all $g,h \in \cg$ are bounded by $M$ to conclude that
\begin{IEEEeqnarray}{rCl}
\big| f_t(g;V_i^*) - f_t(h;V_i^*) \big| 
&  =   &\Big| 2 Y_{i+1}^* \big( g^{[t]}(Z_i^*) - h^{[t]}(Z_i^*) \big) + \big[h^{[t]}(Z_i^*)  + g^{[t]}(Z_i^*)\big] \big[h^{[t]}(Z_i^*)  - g^{[t]}(Z_i^*)\big]   \Big| \nonumber \\
& \leq & (2 Y_{i+1}^* + 2M) \, \big| g^{[t]}(Z_i^*) - h^{[t]}(Z_i^*) \big|\,.  
\end{IEEEeqnarray}
By the contractive property applied to the first component, 
\begin{IEEEeqnarray}{rCl}
&     & \big|g^{[t]}(0,Y_{0}, \ldots, Y_i)-h^{[t]}(0,Y_{0}, \ldots, Y_i)\big| \nonumber  \\ 
&\leq & \big|g\big(g^{[t-1]}(0,Y_{0}, \ldots, Y_{i-1}),Y_i \big) - g \big( h^{[t-1]}(0,Y_{0}, \ldots, Y_{i-1}) , Y_i \big) \big| \nonumber \\
&     & + \big|g\big(h^{[t-1]}(0,Y_{0}, \ldots, Y_{i-1}),Y_i \big) - h \big( h^{[t-1]}(0,Y_{0}, \ldots, Y_{i-1}) , Y_i \big) \big| \nonumber \\
&\leq & L_1 \big|g^{[t-1]}(0,Y_{0}, \ldots, Y_{i-1}) - h^{[t-1]}(0,Y_{0}, \ldots, Y_{i-1})\big| + \| g- h \|_\infty \,,
\end{IEEEeqnarray}
and by an iteration of this argument, we see that $\big| g^{[t]}(Z_i^*) - h^{[t]}(Z_i^*) \big| \leq \|g-h\|_\infty \sum_{k=0}^t L_1^k$.

\item For a metric space $(X,d)$, we denote by $B_d(x,\varepsilon)$ the closed ball around $x \in X$ with radius $\varepsilon$ in the metric $d$. We call the quantity 
\begin{align}
N(\varepsilon, d, X) := \min \Big\{ N \in \Nbb: \, \exists \, x_1, \ldots, x_N \in X \text{ such that} \ X \subset \bigcup_{i = 1}^N B_d(x_i,\varepsilon) \Big\}
\end{align}
the covering number of the metric space $(X,d)$ for the resolution level $\varepsilon$. For $B \in \Nbb_+$ and $L>0$, suppose that $\tilde{\cg}$ is the class of functions defined by 
\begin{IEEEeqnarray}{rCl}
& &\tilde{\cg} := \Big\{ g = (g_0, \ldots, g_{B-1})' \colon [0,M]^B \to [0,M] \,; \nonumber   \\
& & \hspace*{3cm} |g_i(x) - g_i(y)| \leq L_1 \, |x-y| \ \text{ for all } x,y \in [0,M] \Big\} \,. 
\end{IEEEeqnarray}
A straightforward extension of a result by \cite{Kolmogorov1993} regarding covering numbers of classes of Lipschitz functions yields, $\log N(\varepsilon, \|\cdot\|_\infty, \tilde{\cg}) \leq 2MB/\varepsilon$ \citep[cf.][]{Wechsung2019}. Thus, it takes at most
$N := N\big( \epsilon \, , \, \cg(B_n,M,L_1,L_2) \, , \, \|\cdot\|_\infty \big) \leq e^{2MB_n / \epsilon}$
balls to cover the whole class $\cg(B_n,M,L_1,L_2)$ with $\|\cdot\|_\infty$-balls of radius $\epsilon$. Now let $\cg' \subset \cg(B_n,M,L_1,L_2)$ be an arbitrary subset and assume that the elements $\{h_1, \ldots, h_N\} \subset \cg(B_n,M,L_1,L_2)$ constitute a covering of $\cg(B_n,M,L_1,L_2)$ with such balls.
We want to find a set $\{h'_1, \ldots, h'_N\} \subset \cg'$ that constitutes a covering of $\cg'$ with balls of radius $2\epsilon$.  Since $\{h_1,\ldots,h_N\} \subset \cg(B_n,M,L_1,L_2)$ constitutes a covering of $\cg(B_n,M,L_1,L_2)$ with $\epsilon$-balls, we can select a minimal subset $\{h_{i_1}, \ldots, h_{i_{N'}}\} \subset \{h_1, \ldots, h_N\}$ with $1 \leq i_1 < \ldots < i_{N'} \leq N$ that constitutes an $\epsilon$-ball covering of $\cg'$. If all $h_{i_j} \in \cg'$, everything is proven. In this case, the set $\{h_{i_1}, \ldots, h_{i_{N'}}\} \subset \cg'$ constitutes also a $2\epsilon$-ball covering of $\cg'$. Otherwise, assume that $h_{i_j} \in \cg(B_n,M,L_1,L_2)\setminus \cg'$. Since the subset $\{h_{i_1}, \ldots, h_{i_{N'}}\}$ is without loss of generality assumed to be minimal, the set $B(h_{i_j}, \varepsilon) \cap \cg'$ is non-empty. Now pick an arbitrary $h'_{i_j} \in  B(h_{i_j} , \epsilon) \cap \cg'$, and observe that $B(h_{i_j}, \varepsilon) \cap \cg' \subset B(h'_{i_j}, 2\epsilon)\,.$
We can carry out this procedure for every $h_{i_j} \notin \cg'$ and replace this element with the obtained $h'_{i_j}$. All $h_{i_j}$ that are elements of $\cg'$ in the first place are simply relabelled $h'_{i_j}$. Hence, there exists a set $\{h'_{i_1}, \ldots, h'_{i_{N'}}\} \subset \cg'$ that constitutes a cover of $\cg'$ with balls of radius $2\epsilon$. This means that
\begin{align}
N(2\epsilon, \cg', \|\cdot\|_\infty) \leq N' \leq N =  N(\epsilon, \cg(B_n,M,L_1,L_2), \| \cdot \|_\infty) \leq e^{2MB_n/\epsilon} = e^{4MB_n/(2\epsilon)}
\end{align}
for any $k \in \Nbb$.
Since this consideration was independent of the specification of $\cg'$, we conclude that $N(\epsilon, \cg_k, \|\cdot\|_\infty ) \leq e^{4MB_n/\epsilon}$ for all $k$.

Plugging in $\epsilon = 2^{1+k-s}\sqrt{\gamma}\delta$ tells us that
$ N\big( 2^{-s} 2^{k+1} \sqrt{\gamma} \delta, \cg_k, \|\cdot\|_\infty \big) \leq \exp\big({2 M B_n 2^{s-k}\big/(\sqrt{\gamma}\delta)}\big)$.
Thus, there exists a set $\cg_k^{(s)} \subset \cg_k$ with at most $\exp\big({2MB_n2^{s-k} \big/ (\sqrt{\gamma} \delta)}\big)$ elements, and for any $g\in \cg$ there exists an element $h \in \cg_k^{(s)}$ with $\|g-h\|_\infty < 2^{-s} 2^{k+1} \sqrt{\gamma} \delta$. Let now $g \in \cg_k$ be arbitrary. Since $\cg^{(s)}_k$ is finite, the set 
\begin{IEEEeqnarray}{rCl}
\Pi_{s,k}(g)
& := &  \argmin_{h \in \cg_k^{(s)}} \|g-h\|_\infty   =  \big\{ h' \in \cg_k^{(s)} \colon \|g-h'\|_\infty \leq \|g-h\| \ \text{ for all } h \in \cg_k^{(s)} \big\} 
\end{IEEEeqnarray}
is not empty. Choose a representative from the finite set $\Pi_{s,k}(g)$ and call it $g_{s,k}$. \qedhere
\end{enumerate}
\end{proof}

\begin{proof}[Proof of Lemma \ref{lemma:variance}]
Recall that
$\cov (\xi,\eta) \leq \|\eta\|_{L_2(P)} \|\xi\|_{L_2(P)}$ for any two random variables $\xi,\eta$ on $(S,\Sigma,P)$. Thus, from $N < \frac{n-t}{q}$ and because the blocks $\big\{X^*_{2j}(g) \colon j = 0,\ldots,N-1\big\}$ are i.i.d. and centered, we infer,
\begin{IEEEeqnarray}{rCl}
         E \Bigg[ \frac{q}{n-t} \sum_{j=0}^{N-1} X^*_{2j}(g) \Big]^2 
& =   & \frac{q^2}{(n-t)^2} \sum_{j_1,j_2 = 0}^{N-1} E \, X^*_{2j_1}(g)X^*_{2j_2}(g) \leq    \frac{q}{n-t} E \left[ X^*_{0}(g) \right]^2 
\nonumber \\
&  =  &  \frac{q}{n-t} \frac{1}{q^2} \sum_{i_1,i_2=0}^{q-1} \cov(f_t(g;V^*_{i_1+t}),f_t(g;V^*_{i_2+t})) \nonumber \\
&\leq & \frac{q}{n-t} E \big(f_t(g;V_t^*)\big)^2 \,. \label{eq:cov}
\end{IEEEeqnarray}
Let us recall that $Z_t^* = \big( 0,Y_0^*, \ldots, Y_t^*\big)$ is measurable with respect to the $\sigma$-field $\cf_t^* := \sigma\{Y_s^* \colon s \leq t\}$. Therefore, 
\begin{IEEEeqnarray}{rCl}
&     &E \Big[ E_{|\cf_{t}^*} \sup_{\alpha \in [0,M]}\Big[\big( Y_{t+1}^* - \alpha \big) \,  \big( m^{[t]}(Z_t^*) - g^{[t]}(Z_t^*) \big) \Big]^2 \Big] 
\nonumber \\
&  =  &E \Big[ \big( m^{[t]}(Z_t^*) - g^{[t]}(Z_t^*) \big)^2 \, E_{|\cf_{t}^*}\Big[ \sup_{\alpha \in [0,M]} \big(Y_{t+1}^* - E_{|\cf_{t}^*} Y_{t+1}^* + E_{|\cf_{t}^*} Y_{t+1}^* - \alpha \big)^2  \Big] \Big] \nonumber \\
&\leq &E \Big[ \big( m^{[t]}(Z_t^*) - g^{[t]}(Z_t^*) \big)^2 \Big[ E_{|\cf_{t}^*} \Big( Y^*_{t+1} - E_{|\cf_{t}^*} Y^*_{t+1} \Big)^2 + \sup_{\alpha \in [0,M]}\big(  
 E_{|\cf_{t}^*} Y_{t+1}^* - \alpha \big)^2 \nonumber \\
&     & \hspace{3.5cm} + 2\, \sup_{\alpha \in [0,M]} \Big( \big| E_{|\cf_{t}^*} Y_{t+1}^* - \alpha \big|  \Big) \Big( E_{|\cf_{t}^*} \big| Y^*_{t+1} - E_{|\cf_{t}^*} Y^*_{t+1} \big| \Big) \Big] \Big]\nonumber \\
&\leq & E \Big[ \big( m^{[t]}(Z_t^*) - g^{[t]}(Z_t^*) \big)^2 \Big[	\var_{|\cf^*_{t}} \big( Y^*_{t+1}\big)   +  M^2 + 2\, M \, \sqrt{\var_{|\cf^*_{t}} \big( Y^*_{t+1}\big)} \Big] \Big] \nonumber \\
&\leq & (M^2 + 2\, M^{3/2} + M) \, E\big( m^{[t]}(Z_t^*) - g^{[t]}(Z_t^*) \big)^2 \,. \label{eq:alpha}
\end{IEEEeqnarray}
Since $\frac{1}{2}(m^{[t]}(Z_t^*) + g^{[t]}(Z_t^*)) \in [0,M]$, 
\begin{IEEEeqnarray}{+rCl+x*}
E \big( f_t(g;V_t^*) \big)^2 
& =   &  4 \ E \Bigg[ \Big( Y_{t+1}^* - \frac{m^{[t]}(Z_t^*) + g^{[t]}(Z_t^*)}{2}\Big) \big( m^{[t]}(Z_t^*) - g^{[t]}(Z_t^*) \big) \Bigg]^2 \nonumber \\
&\leq &4 \ E \sup_{\alpha \in [0,M]} \Big[ \big(Y_{t+1}^* - \alpha \big)\big( m^{[t]}(Z_t^*) - g^{[t]}(Z_t^*) \big) \Big]^2 \label{eq:second_moment} \\
&\leq & 4\,\big(M^2 + 2\,M^{3/2} + M\big) E\big( m^{[t]}(Z_t^*) - g^{[t]}(Z_t^*) \big)^2\,. & \qedhere \nonumber
\end{IEEEeqnarray}
\end{proof}

\begin{proof}[Proof of Lemma \ref{lemma:chaining}]
 In accordance with Lemma \ref{lemma:covering} \ref{prop:g_s}, for $g \in \cg(B_n,M,L_1,L_2)$ and $\check{S}(n) := \min \big\{ s \in \Nbb \colon \frac{4}{1-L_1} 2^{-s} \sqrt{\gamma} \delta \leq 2^{-6} \gamma \delta^2 /(15\,M) \big\}$, let the functions $g_{0,k}, \ldots, g_{\check{S},k}$ satisfy 
\begin{align}
\|g-g_{s,k}\|_\infty \leq 2^{-s}2^{k+1} \sqrt{\gamma} \delta\,.
\end{align} 
For every $n \in \Nbb$ the maximal index $\check{S} = \check{S}(n)$ can be written alternatively as $\check{S} = \min \big\{ s \in \Nbb \colon \frac{2}{1-L_1} 2^{-s}2^{k+1} \sqrt{\gamma} \delta \leq 2^{k-6} \gamma \delta^2 /(15\,M) \big\}$. 
 
Using a telescope representation of $g \in \cg(B_n,M,L_1,L_2)$, linearity of
$g\mapsto X_{2j}^*(g)$, and the triangle inequality for probabilities, we obtain
\begin{IEEEeqnarray}{rCl}
&       &P \Bigg\{ \sup_{g \in \cg_k} \frac{q}{n-t}  \sum_{j=0}^{N-1} \varepsilon_j  X^*_{2j}(g)  
              > 2^{2k-6} \gamma \delta^2 \Bigg\} \nonumber \\
& =     & P \Bigg\{ \sup_{g \in \cg_k} \frac{q}{n-t} \sum_{j=0}^{N-1} \varepsilon_j  \Big[  X^*_{2j}(g) - X^*_{2j}(g_{\check{S},k}) + X^*_{2j}(g_{0,k}) \nonumber \\
&       & \qquad + \sum_{s=0}^{\check{S}-1} \big(X^*_{2j}(g_{s+1,k}) - X^*_{2j}(g_{s,k}) \big) \Big]    > 2^{2k-6} \gamma \delta^2 \Bigg\}
\nonumber \\
& \leq  & P \Bigg\{ \sup_{g \in \cg_k} \frac{q}{n-t} \sum_{j=0}^{N-1} \varepsilon_j \big[   X^*_{2j}(g) - X^*_{2j}(g_{\check{S},k}) \big] 
              > 2^{2k-6} \gamma \delta^2 / 3 \Bigg\} \nonumber \\
&       & \qquad + P \Bigg\{ \sup_{g \in \cg_k} \frac{q}{n-t} \sum_{j=0}^{N-1}  \varepsilon_j   X^*_{2j}(g_{0,k}) 
              > 2^{2k-6} \gamma \delta^2 / 3 \Bigg\} \nonumber \\ 
&       & \qquad + P \Bigg\{ \sup_{g \in \cg_k} \sum_{s=0}^{\check{S}-1} \frac{q}{n-t} \sum_{j=0}^{N-1}  \varepsilon_j  \big(X^*_{2j}(g_{s+1,k}) - X^*_{2j}(g_{s,k}) \big) 
              > 2^{2k-6} \gamma \delta^2 / 3 \Bigg\} \nonumber \\
&  =:   & P_1 + P_2 + P_3\,. 
\end{IEEEeqnarray}
We treat each of the three terms separately. As to the first term, the index $\check{S}$ was chosen such that the approximation of $g$ by $g_{\check{S},k}$ is very accurate. Using the definition of $X^*_{2j}(g)$, Lemma \ref{lemma:covering} \ref{lemma:f_g}, and the definition of the sequence $\{g_{s,k}\}_{s=0, \ldots , \check{S}}$, we observe
\begin{IEEEeqnarray}{rCl}
         \big|  X^*_{2j}(g) - X^*_{2j}(g_{\check{S},k}) \big| 
& =    & \frac{1}{q} \Big| \sum_{i=0}^{q-1} \Big( 
          f_t(g;V^*_{t+2jq+i}) - 
          f_t({g_{\check{S},k}};V^*_{t+2jq+i}) \nonumber  \\
&      & \qquad \qquad \qquad 
            -  E \,  \big[ f_t(g;V^*_{t+2jq+i}) - f_t({g_{\check{S},k}};V^*_{t+2jq+i}) \big]  \Big) \Big| \nonumber \\         
&\leq  & \frac{1}{q} \sum_{i=0}^{q-1} \Big( 
         \underbrace{\big| f_t(g;V^*_{t+2jq+i}) - f_t({g_{\check{S},k}};V^*_{t+2jq+i}) \big|}_{\leq 2\|g-g_{\check{S},k}\|_\infty(Y^*_{t+2jq+i+1}+M)/(1-L_1)}  \nonumber \\
&      & \qquad \qquad \qquad 
            + \underbrace{ E \,  \big| f_t(g;V^*_{t+2jq+i}) - f_t({g_{\check{S},k}};V^*_{t+2jq+i}) \big| }_{\leq 2\|g-g_{\check{S},k}\|_\infty(EY^*_{t+2jq+i+1}+M)/(1-L_1)} \Big) \nonumber \\
& \leq & \frac{1}{q} \sum_{i=0}^{q-1} (Y^*_{t+2jq+i+1} + 3M) \frac{2}{1-L_1} 2^{-\check{S}} 2^{k+1} \sqrt{\gamma} \delta  \nonumber \\            
& \leq & \frac{1}{q} \sum_{i=0}^{q-1} (Y^*_{t+2jq+i+1} + 3M) 2^{k-6} \gamma \delta^2 / (15 M) \,.
\end{IEEEeqnarray} 
In the last estimate follows from the definition of
$\check{S}$. We conclude, 
\begin{IEEEeqnarray}{rCl}
&      & P \Bigg\{ \sup_{g \in \cg_k} \frac{q}{n-t} \Big|\sum_{j=0}^{N-1} \varepsilon_j \big[   X^*_{2j}(g) - X^*_{2j}(g_{\check{S},k}) \big] \Big| 
           > 2^{2k-6} \gamma \delta^2 / 3 \Bigg\} \nonumber \\ 
& \leq & P \Bigg\{ \frac{1}{n-t} \Big| \sum_{i=t}^{n-1} (Y_{i+1}^* + 3M) 2^{k-6}                			\gamma \delta^2 / (15 \, M) \Big| 
           >    2^{2k-6} \gamma \delta^2 / 3 \Bigg\} \nonumber \\
& \leq &P \Bigg\{ \frac{1}{n-t}\Big| \sum_{i=t}^{n-1}  Y_{i+1}^*\Big| >  2\,M\,2^{k} \Bigg\} \nonumber \\
& \leq &P \Bigg\{ \frac{1}{n-t}\Big| \sum_{i=t}^{n-1} (Y^*_{i+1} - E\,
Y^*_{i+1})\Big| > M \, 2^k \Bigg\} \nonumber \\
& \leq & \frac{2^{-2k} }{M^2\, (n-t)^2} \var \big( \sum_{i=t}^{n-1} Y^*_i \big)\,,
\end{IEEEeqnarray}
for any $k$. We recall that the process $\{Y_i^*\}$ is $q$-dependent and stationary and conclude that 
\begin{IEEEeqnarray}{rCl}
\var \big( \sum_{i=0}^{n-1} Y_i^* \big)
&  =  & \sum_{i,j=0}^{n-1} \cov \big( Y_i^*, Y_j^* \big) 
   =   \sum_{\substack{0 \leq i,j \leq n-1 \\ |i-j| \leq q}} \big( E(Y_i^*Y_j^*) - EY_i^*EY_j^* \big) \nonumber \\
&\leq & 2 \sum_{r=0}^{q-1} \sum_{i=1}^{n-r} \Big( \sqrt{EY_i^*\,^2}\sqrt{EY_{i+r}^*\,^2} - EY_i^*EY_{i+r}^* \Big) \nonumber  \\
&\leq & 2 n q \Big(E\big[Y_0^*\,^2\big] - \big[EY_0^*\big]^2\Big)
 \leq   2M \, nq\,. 
\end{IEEEeqnarray}
This proves that there exists a constant $C > 0$ and a number $n^{(P_1)}$ such that for all $n \in \Nbb$ with $n > n^{(P_1)}$ and all $k \in \Nbb$
\begin{align}
 P_1 = P \Bigg\{ \sup_{g \in \cg_k} \frac{q}{n-t} \Big|\sum_{j=0}^{N-1} \varepsilon_j \big[   X^*_{2j}(g) - X^*_{2j}(g_{\check{S},k}) \big] \Big| > 2^{2k-6} \gamma \delta^2 / 3 \Bigg\} \leq C \ 2^{-2k} \, \frac{q}{n-t}.
\end{align}

We proceed by addressing the second term, $P_2$. First of all, since for any $g \in \cg_k$ the first approximation $g_{0,k} = \pi_{0,k}g$ is selected from the finite set $\cg_k^{(0)}$, 
\begin{IEEEeqnarray}{rCl}
&    &P  \Bigg\{ \sup_{g \in \cg_k} \frac{q}{n-t} \sum_{j=0}^{N-1}  \varepsilon_j   X^*_{2j}(g_{0,k})   
				> 2^{2k-6} \gamma \delta^2 / 3 \Bigg\} \nonumber \\ 
&    & \hspace{2cm} = P  \Bigg\{ \max_{h_{0,k} \in \cg_k^{(0)}} \frac{q}{n-t} \sum_{j=0}^{N-1}  \varepsilon_j   X^*_{2j}(h_{0,k})  
				> 2^{2k-6} \gamma \delta^2 / 3 \Bigg\} \,.
\end{IEEEeqnarray}
This exceedance probability will be bounded with the help of Bernstein's inequality for sums of bounded random variables 
\citep[p.118]{gine2015}. To that end, we introduce the $\Sigma$-measurable events
\begin{IEEEeqnarray}{rCl}
A_{k,2j} = A_{k,2j}(n) := \big\{ \omega \in S \colon \max_{i =0, \ldots, q-1} Y^*_{t+2jq+i+1} \leq 2(k+2) \log n \big\} \,.
\end{IEEEeqnarray}
In order to bound the probability of these events, we use an exponential tail bound for Poisson variables 
\citep[p. 116]{gine2015} and the fact that $|\lambda_i|$ is uniformly bounded by $M$. Thence we conclude,
\begin{IEEEeqnarray}{+rCl+x*}
\pr \Big\{ \max_{1 \leq i \leq n} Y_i > 2(k+2) \, \log n \Big\} 
& \leq & \sum_{i=1}^n \ew \Big[ \pr \big\{ Y_i - \lambda_i > 2(k+2) \, \log(n) - \lambda_i \ \big| \ \lambda_i \big\} \Big] \nonumber \\
& \leq & \sum_{i=1}^n\, \ew \Big[ \exp \Big( - \frac{\big(2(k+2) \log(n) - \lambda_i\big)^2}{2 \lambda_i + \frac{2}{3} (2(k+2) \log(n) - \lambda_i)}  \Big) \Big] \nonumber \\
& \leq &  \sum_{i=1}^n \ew \Big[ \exp \Big( - {(k+2) \log(n) + \lambda_i} \Big) \Big] \nonumber \\
& \leq & \exp \Big( \log(n) - (k+2) \log(n) + M  \Big) \nonumber \\
& =    & e^M \, n^{-(k+1)}\,.
\end{IEEEeqnarray}
Consequently, $P \Big( \bigcup_{j=0}^{N-1} A_{k,2j}^c \Big) \leq e^M n^{-(k+1)}$, which implies
\begin{IEEEeqnarray}{rCl}
&          & P  \Bigg\{ \max_{h_{0,k} \in \cg_k^{(0)}} \frac{q}{n-t} \sum_{j=0}^{N-1}  \varepsilon_j   X^*_{2j}(h_{0,k})  
				> 2^{2k-6} \gamma \delta^2 / 3 \Bigg\} \nonumber \\
&   \leq   & e^M \, n^{-(k+1)} + P  \Bigg\{ \max_{h_{0,k} \in \cg_k^{(0)}} \frac{q}{n-t} \sum_{j=0}^{N-1}  \varepsilon_j   X^*_{2j}(h_{0,k}) \ind_{A_{k,2j}}  
				> 2^{2k-6} \gamma \delta^2 / 3 \Bigg\} \nonumber \\				
&  \leq    & e^M \, n^{-(k+1)} +  \sum_{h_{0,k} \in \cg_k^{(0)}}   
 			 P \Bigg\{ \frac{q}{n-t} \sum_{j=0}^{N-1}  \varepsilon_j   X^*_{2j}(h_{0,k}) \ind_{A_{k,2j}}
                > 2^{2k-6} \gamma \delta^2 / 3 \Bigg\} \,.	           
\end{IEEEeqnarray}
Now all involved variables are bounded. In order to apply Bernstein's inequality, we need bounds on the variance of the sum $\sum_{j=0}^{N-1}  \varepsilon_j   X^*_{2j}(h_{0,k}) \ind_{A_{k,2j}}$, and a bound on the absolute values of the addends $\varepsilon_j   X^*_{2j}(h_{0,k}) \ind_{A_{k,2j}}$. Furthermore, the addends have to be centred.
As for the variance bound, recall that the sequences $\{\varepsilon_j\}$ and $\{Y_i^*\}$ are independent. Hence, $\big\{ \varepsilon_j   X^*_{2j}(h_{0,k})  \ind_{A_{k,2j}} \big\}_{j \in \Nbb}$ is a sequence of i.i.d. random variables, and $E \big[ \varepsilon_j   X^*_{2j}(h_{0,k})  \ind_{A_{k,2j}} \big] = E \varepsilon_j \,  E \big[X^*_{2j}(h_{0,k})  \ind_{A_{k,2j}} \big]  = 0\,.$  Thus, for any $h_{0,k} \in \cg_k^{(0)} \subset \cg_k$, we can invoke Lemma \ref{lemma:variance} \ref{i} to conclude that there exists a number $n^*$ such that for all $n > n^*$ and for all $k \in \Nbb$
\begin{IEEEeqnarray}{rCl}
\var \Big( \sum_{j=0}^{N-1}  \varepsilon_j   X^*_{2j}(h_{0,k})  \ind_{A_{k,2j}} \Big) 
&   \leq  & \sum_{j=0}^{N-1}  E \big( X^*_{2j}(h_{0,k}) \big)^2 \nonumber \\
&   =     & \frac{(n-t)^2}{q^2} \var \left( \frac{q}{n-t} \sum_{j=0}^{N-1} X^*_{2j}(h_{0,k}) \right) \nonumber \\
&\leq     & \frac{(n-t)^2}{q^2} (M^2+2\,M^{3/2}+M)  \, 2^{2k+4} \gamma \delta^2  \frac{q}{n-t} \nonumber \\
& =       & C_1 \frac{(n-t)}{q} \, 2^{2k} \delta^2 \nonumber \\
& :=      & \sigma_n^2\,,
\end{IEEEeqnarray}
with $C_1 := 16\,(M^2+2\,M^{3/2}+M)\, \gamma$. Let us now bound the absolute values of the addends. Recall that $f_t(m;V^*_{t+2jq+i}) = 0$ and $\|g-m\|_\infty \leq M$ for all $g \in \cg$, and infer with Lemma \ref{lemma:covering} \ref{lemma:f_g} that
\begin{IEEEeqnarray}{rCl}
\big| f_t(h_{0,k};V^*_{t+2jq+i}) \big| 
 \leq  \frac{2 \|h_{0,k}-m\|_\infty}{1-L_1} (M + Y^*_{t+2jq+i+1}) 
 \leq  \frac{2 M}{1-L_1} (M + Y^*_{t+2jq+i+1}) \,.
\end{IEEEeqnarray}
In virtue of $E Y_{j} \leq M$ and the definition of the events $A_{k,2j}$, we obtain 
\begin{IEEEeqnarray}{rCl}
\big| \varepsilon_j X^*_{2j}(h_{0,k}) \ind_{A_{k,2j}} \big| 
& \leq &  \ind_{A_{k,2j}} \ \frac{1}{q}  \sum_{i=0}^{q-1} \Big( \big|  f_t(h_{0,k} ; V^*_{t+2jq+i}) \big| + E \big| f_t(h_{0,k} ; V^*_{t+2jq+i}) \big| \Big) \nonumber \\
& \leq & \frac{1}{q} \sum_{i=0}^{q-1} \frac{2M}{1-L_1}  \big( Y^*_{t+2jq+i+1} + 3M \big) \, \ind_{A_{k,2j}} \nonumber \\
& \leq & \frac{2M}{1-L_1} \big(2(k+2)\log(n) + 3M\big) \nonumber \\
& \leq & C_2 \, (k+1) \log n \nonumber\\
& =:   & b_n
\end{IEEEeqnarray} 
with $C_2 = \frac{10 \, M}{1-L_1}$, for all $n \geq e^{3\,M}$ and $k \in \Nbb$. We are ready to apply Bernstein's inequality. Introducing the variables
\begin{IEEEeqnarray}{cCc}
\eta_j  := \varepsilon_j X_{2j}^*(h_{0,k}) \ind_{A_{k,2j}} 
 $ \hspace{1cm} \text{and} \hspace{1cm} $
 x_n    :=  \frac{n-t}{q} 2^{2k-6} \gamma \delta^2/3\,,
\end{IEEEeqnarray}
we obtain the display
\begin{IEEEeqnarray}{rCl}
&        &  P \Bigg\{  \sum_{j=0}^{N-1}  \varepsilon_j   X^*_{2j}(h_{0,k}) \ind_{A_{k,2j}}
                > \frac{n-t}{q} 2^{2k-6} \gamma \delta^2 / 3 \Bigg\} 
    =     P \Bigg\{ \sum_{j=0}^{N-1} \eta_j > x_n \Bigg\} \,.
\end{IEEEeqnarray}
We have shown that the random variables $\eta_j$ are independent and centred, $|\eta_j| \leq b_n$, and $\var(\eta_0 + \ldots + \eta_{N-1}) \leq \sigma_n^2$. Bernstein's inequality yields,
\begin{IEEEeqnarray}{rCl}\label{bernsteinP2}
P \Bigg\{ \sum_{j=0}^{N-1} \eta_j > x_n \Bigg\} 
 \leq    \exp \Bigg( - \frac{1}{2} \frac{x^2}{\sigma_n^2 + x_n  b_n / 3}  \Bigg) \,.
\end{IEEEeqnarray}
In this case, $\sigma_n^2 \asymp 2^{2k} \frac{n-t}{q} \delta^2$ is dominated by $x_n\,b_n \asymp  2^{2k}\frac{n-t}{q} \delta^2 \, (k+1) \log n$ since 
\begin{IEEEeqnarray}{rCl}
\limsup_{n \to \infty} \ \sup_{k \in \Nbb} \ \frac{\sigma_n^2}{x_n b_n} \leq \limsup_{n \to \infty} \ \sup_{k \in \Nbb} \frac{C}{k+1} \frac{ 1 }{ \log n} \leq C \limsup_{n \to \infty} \ (\log n)^{-1} = 0
\end{IEEEeqnarray}
with a positive constant $C$. Hence, there exists a number $n^{**}$, independent of $k$, such that $\sigma_n^2 \leq \frac{2}{3} x_n b_n$ for all $n > n^{**}$ and all $k \in \Nbb$. Consequently, under the assumptions $\delta_n =C_\delta \ n^{-1/3} \log n$ and $t(n) \asymp q(n) \asymp \log n$, we obtain for the exponent in \eqref{bernsteinP2},
\begin{IEEEeqnarray}{rCl}
\liminf_{n \to \infty} \frac{1}{2} \, \frac{x_n^2}{\sigma_n^2 + x_n b_n/3} \cdot n^{-1/3}
& \geq  &  \liminf_{n \to \infty} \, \frac{x_n}{2 \, b_n} \cdot n^{-1/3} \nonumber \\
&  =    & \frac{2^{-7}\gamma\,C_\delta}{3\, C_2} \, \frac{2^{2k}}{k+1} \ \liminf_{n \to \infty} \frac{n-t}{q(n)} \frac{n^{-2/3} (\log n)^2}{\log n} \cdot n^{-1/3} \nonumber \\
& \geq  & C_3 \, 2^k
\end{IEEEeqnarray}
for $C_3 = 2^{-7}\gamma\,C_\delta \big/(3\, C_2) \ \liminf_n \log(n)/q(n)  > 0$. Hence, there exists a number $n_0 \geq \max\{ e^{3M} \,,\, n^* \,,\, n^{**}\}$ such that $\frac{1}{2} {x^2}\big/{(\sigma_n^2 + x_n  b_n / 3)} \geq C_3 \, 2^k \, n^{1/3}$ for all $n\geq n_0$ and all $k \in \Nbb$. In conclusion, for all $k$ and all $n \geq n_0$,
\begin{IEEEeqnarray}{rCl}
  P \Bigg\{  \sum_{j=0}^{N-1}  \varepsilon_j   X^*_{2j}(h_{0,k}) \ind_{A_{k,2j}}
                > \frac{n-t}{q} 2^{2k-6} \gamma \delta^2 / 3 \Bigg\}
\leq \exp \Big( -C_3 \, 2^k n^{1/3}  \Big)\,.              
\end{IEEEeqnarray}
Note that $n_0$ does not depend on $k$. Moreover,
the previous bound is independent of the specific function $h_{0,k} \in \cg_k^{(0)}$.

Let $C_B = B_n/\log n$ and note that $C_B = O(1)$. Since $\sup_k \ \log \big( \# \cg_k^{(0)} \big) \leq \sup_k \ C_4\, 2^{-k} \, \log(n) \big/ \delta \leq C_4 \, n^{1/3}$ with
$C_4 := 2MC_B\big/\sqrt{\gamma}$, which follows from the bounds in Lemma \ref{lemma:covering} \ref{prop:g_s} with $s=0$, we conclude,
\begin{IEEEeqnarray}{rRL}
& \sum_{h_{0,k} \in \cg_k^{(0)}}   
 			 P \Bigg\{ \frac{q}{n-t} \sum_{j=0}^{N-1}  \varepsilon_j   X^*_{2j}(h_{0,k}) \ind_{A_{k,2j}}
                > 2^{2k-6} \gamma \delta^2 / 3 \Bigg\} \nonumber 
&  \leq \# \cg_k^{(0)} \
 \exp 
\big( - C_3 \, 2^{k} n^{1/3} \big) \nonumber \\
&  \leq  \exp 
\Big(  (C_4 - C_3)\, 2^k  n^{1/3} \Big)
\end{IEEEeqnarray}
for all $n \geq n_0$ and $k \in \Nbb$. We can choose $C_\delta$ independently of $n$ and $k$ such that $C_3 = 2\, C_4$. Hence,   
\begin{IEEEeqnarray}{rCl}
&        &  \sum_{h_{0,k} \in \cg_k^{(0)}}   
 			 P \Bigg\{ \frac{q}{n-t} \sum_{j=0}^{N-1}  \varepsilon_j   X^*_{2j}(h_{0,k}) \ind_{A_{k,2j}}
                > 2^{2k-6} \gamma \delta^2 / 3 \Bigg\}  
 \leq   \exp \big( - C_4 \, 2^k n^{1/3} \big)
\end{IEEEeqnarray}
for all $n \geq n_0$ and all $k \in \Nbb$.
In conclusion, there exists a natural number $n^{(P_2)} \geq n_0$ such that for all $n \geq n^{(P_2)}$ and all $k \in \Nbb$ the term $P_2$ is bounded by
\begin{IEEEeqnarray}{rCl}
&          & P  \Bigg\{ \sup_{g \in \cg_k} \frac{q}{n-t} \sum_{j=0}^{N-1}  \varepsilon_j   X^*_{2j}(g_{0,k})   
				> 2^{2k-6} \gamma \delta^2 / 3 \Bigg\} 
   <      e^M \, n^{-(k+1)} +  \exp 
\Big( - C'\, 2^{k} \, \, n^{1/3} \Big)\,,
\end{IEEEeqnarray}
with some positive constant $C'$.

It is left to find a bound for the third term, $P_3$. For that sake, we define the sets $\cm_{s,k}$ by
\begin{IEEEeqnarray}{rCl}
\cm_{s,k} := \Big\{ ( g_1, g_2 ) \colon g_1 \in \cg_k^{(s)} \,,\, g_2 \in \cg_k^{(s+1)} \,,\, \|g_1 - g_2 \|_\infty \leq 2^{-s} 2^{k+2} \sqrt{\gamma} \delta \Big\}\,.
\end{IEEEeqnarray}  
By definition of $\{g_{s,k}\colon {s=0,\ldots,\check{S}}\}$, $\|g-g_{s,k}\|_\infty \leq  2^{-s} 2^{k+1} \sqrt{\gamma} \delta
$ as well as $\|g-g_{s+1,k}\|_\infty  \leq  2^{-(s+1)} 2^{k+1} \sqrt{\gamma}\delta$, and by the triangle inequality, $\|g_{s,k}-g_{s+1,k} \|_\infty
 \leq  \|g_{s,k}- g\|_\infty 
         + \|g - g_{s+1,k} \|_\infty 
 \leq  2^{-s} 2^{k+2} \sqrt{\gamma} \delta$.
Hence, $(g_{s,k},g_{s+1,k}) \in \cm_{s,k}$, and we conclude,
\begin{IEEEeqnarray}{rCl}
&   &\sup_{g \in \cg_k} \sum_{s=0}^{\check{S}-1} \frac{q}{n-t} \Big|\sum_{j=0}^{N-1}  \varepsilon_j \big(X^*_{2j}(g_{s+1,k}) - X^*_{2j}(g_{s,k}) \big) \Big|  
   =    \max_{(g_1,g_2) \in \cm_{k,s}} \sum_{s=0}^{\check{S}-1} \frac{q}{n-t} \Big|\sum_{j=0}^{N-1}  \varepsilon_j  \big(X^*_{2j}(g_{2}) - X^*_{2j}(g_{1}) \big) \Big| \nonumber  \\
&   &\hspace{4cm} \leq     \sum_{s=0}^{\check{S}-1}  \Bigg[ \max_{(g_1,g_2) \in \cm_{k,s}}  \frac{q}{n-t} \Big|\sum_{j=0}^{N-1}  \varepsilon_j  \big(X^*_{2j}(g_{2}) - X^*_{2j}(g_{1}) \big) \Big| \Bigg] \,.  
 \label{eq:inf-max}
\end{IEEEeqnarray} 
Using again Lemma \ref{lemma:covering} \ref{lemma:f_g}, we observe for $(g_1,g_2) \in \cm$, 
\begin{IEEEeqnarray}{rCl}
\Big| \,\big(X^*_{2j}(g_{2}) - X^*_{2j}(g_{1}) \big) \Big| 
&\leq & \frac{1}{q}  \sum_{i=0}^{q-1}   \big| f_t({g_1};V_{t+2jq+i}^*) - f_t({g_2};V_{t+2jq+i}^*) \big| \nonumber\\
&     & + \frac{1}{q}  \sum_{i=0}^{q-1}   E\big| f_t({g_1};V_{t+2jq+i}^*) - f_t({g_2};V_{t+2jq+i}^*) \big| \nonumber\\
&\leq &   \frac{2^{-s} 2^{k+3} \sqrt{\gamma} \delta}{1-L_1}   \frac{1}{q}  \sum_{i=0}^{q-1}  (Y_{t+2jq+i+1}^* + 3\, M )  \,.
\end{IEEEeqnarray}
Since the variables $\varepsilon_j$ are independent and centred, we can apply Hoeffding's inequality 
\citep[p.114]{gine2015} conditionally on $\big( V_t^* , \ldots, V_{n-1}^* \big)'$ and obtain for $(g_1,g_2) \in \cm_{k,s}$,
\begin{IEEEeqnarray}{rCl}
&       & P \Bigg\{ \frac{q}{n-t} \Big|\sum_{j=0}^{N-1}  \varepsilon_j \big(X^*_{2j}(g_{2}) - X^*_{2j}(g_{1}) \big) \Big| > x \, \Bigg| \, V_t^* , \ldots , V_{n-1}^*  \Bigg\} \nonumber \\
& \leq  & 2\, \exp \left( - \frac{1}{2}\frac{x^2} 
{\frac{q^2}{(n-t)^2} \sum_{j=0}^{N-1}   \big(X^*_{2j}(g_{2}) - X^*_{2j}(g_{1}) \big)^2   } \right) \nonumber \\
& \leq & 2\, \exp \left( - \frac{ x^2}{2 \sum_{j=0}^{N-1}  \Big[ \frac{2^{-s} 2^{k+3} \sqrt{\gamma} \delta}{(n-t)(1-L_1)} \sum_{i=0}^{q-1} \big(Y^*_{t+2jq+i+1} + 3\, M \big) \Big]^2   } \right)\,.
\end{IEEEeqnarray}
We apply an integrated Bernstein-type inequality 
\citep[p.408]{Doukhan95}. For random variables $\xi_1, \ldots,\xi_m$ that satisfy the tail bound $P\{|\xi_i| > x \} \leq 2 \, e^{-\frac{1}{2} \frac{x^2}{b+ax}}$ for some $a,b \geq 0$ and all $x \geq 0$, the inequality states, 
\begin{IEEEeqnarray}{rCl}
E \, \big[\max_{i \leq m} |\xi_i| \big] \leq C \big( \sqrt{ b \log m } + a \log m \big)
\end{IEEEeqnarray}
for some universal constant $C > 0$. We apply this inequality with $a=0$ and
\begin{IEEEeqnarray}{rCl}
b =  \sum_{j=0}^{N-1}  \Big[ \frac{2^{-s} 2^{k+3} \sqrt{\gamma} \delta}{(n-t)(1-L_1)} \sum_{i=0}^{q-1} \big(Y^*_{t+2jq+i+1} +3\, M \big) \Big]^2\,.
\end{IEEEeqnarray}
Conditionally on $V^*_t, \ldots, V^*_{n-1}$, this yields almost surely in $P$,
\begin{IEEEeqnarray}{rCl}
&      & E \Bigg[ \max_{(g_1,g_2) \in \cm_{k,s}}  \frac{q}{n-t} \Big|\sum_{j=0}^{N-1}  \varepsilon_j \big(X^*_{2j}(g_{2}) - X^*_{2j}(g_{1}) \big) \Big| \, \Bigg| V_t^*, \ldots, V_{n-1}^* \Bigg] \nonumber  \\
& \leq & C \sqrt{\log \# \cm_{k,s}} \ \sqrt{  \sum_{j=0}^{N-1}  \Big[ \frac{2^{-s} 2^{k+3} \sqrt{\gamma} \delta}{(n-t)(1-L_1)} \sum_{i=0}^{q-1} \big(Y^*_{t+2jq+i+1} +3 M \big) \Big]^2 }
\label{eq:integrated_bs}
\end{IEEEeqnarray}
for some positive constant $C$. Furthermore, as $E_{|\cf^*_i}Y_{i+1}^{*\,2} = \var_{|\cf^*_i} Y^*_{i+1} + (E_{|\cf^*_i} Y^*_{i+1})^2 \leq M +M^2$, we conclude invoking the triangle inequality for the $L_2(P)$ norm,
\begin{IEEEeqnarray}{rCl}
       \Bigg( E \Bigg[\sum_{i=0}^{q-1} \big( Y^*_{t+2jq+i+1} + 3 \, M \big) \Bigg]^2 \Bigg)^{1/2} \nonumber 
& \leq & \sum_{i=0}^{q-1} \Big( E \big[ Y^*_{t+2jq+i+1} + 3 \, M \big]^2  \Big)^{1/2} 
  =      q \ E  \big(  Y^*_{0} + 3 \, M \big)^2 \\
& \leq & 16 \, (M^2 +M ) \, q  
 =:    \sqrt{\tilde{C}} \, q\,.  \label{eq:q} 
\end{IEEEeqnarray}
Recall that $N < \frac{n-t}{q}$ and that $x\mapsto\sqrt{x}$ is a concave function. Define $C_5 := C \,  \frac{\sqrt{ 64 \, \tilde{C} \, \gamma} }{1-L_1}$ and integrate inequality \eqref{eq:integrated_bs} with respect to $P$. This yields,
\begin{IEEEeqnarray}{rClr}
&       & E \Bigg[ E \Bigg[ \max_{(g_1,g_2) \in \cm_{k,s}}  \frac{q}{n-t} \Big|\sum_{j=0}^{N-1}  \varepsilon_j \big(X^*_{2j}(g_{2}) - X^*_{2j}(g_{1}) \big) \Big| \, \Bigg| V_t^*, \ldots, V_{n-1}^* \Bigg] \Bigg] & \nonumber \\
& \leq  & C \sqrt{\log \# \cm_{k,s}} 
\sqrt{\frac{64 \, \gamma}{(1-L_1)^2} \sum_{j=0}^{N-1} \Bigg[ \frac{2^{-s+k}\delta}{(n-t)} \Bigg]^2  E \Bigg[ \sum_{i=0}^{q-1} \big( Y^*_{t+2jq+i+1} + 3M \big) \Bigg]^2}
                        & \nonumber \\
& \leq     & C_5 \sqrt{\log \# \cm_{k,s}} \frac{2^{k-s} \delta \sqrt{q}}{\sqrt{n-t}}
      \label{eq:EW-max} &
\end{IEEEeqnarray}
Concerning the cardinality of the set $\cm_{k,s}$, we observe that
\begin{IEEEeqnarray}{rCl}
\# \cm_{k,s} \leq \# \cg_k^{(s)} \# \cg_k^{(s+1)} \leq e^{2MB_n \, (2^{s-k} + 2^{s+1-k})/(\sqrt{\gamma}\delta)} \leq e^{2MB_n\, 2^{s+2-k}/ (\sqrt{\gamma} \delta)}\,,
\end{IEEEeqnarray}
or, with $C_6 := \sqrt{\frac{8 M C_B}{\sqrt{\gamma}}}$,
\begin{IEEEeqnarray}{rCl}
\sqrt{ \log \# \cm_{k,s}} \leq C_6 \ 2^{s/2} 2^{-k/2} n^{1/6}\,, \label{eq:M}
\end{IEEEeqnarray}
where again $C_B = B_n/\log n$. Applying Markov's inequality, we finally arrive at a bound for $P_3$:
\begin{IEEEeqnarray}{rClr}
&         &  P \Bigg\{ \sup_{g \in \cg_k}  \sum_{s=0}^{\check{S}-1} \frac{q}{n-t} \Big|\sum_{j=0}^{N-1}  \varepsilon_j  \big(X^*_{2j}(g_{s+1,k}) - X^*_{2j}(g_{s,k}) \big) \Big| > 2^{2k-6} \gamma \delta^2 / 3 \Bigg\} 
                      & \nonumber \\
& \leq    & \frac{3 \cdot 2^6}{\gamma} 2^{-2k} \delta^{-2} \ E \Bigg[\  \sup_{g \in \cg_k} \ \sum_{s=0}^{\check{S}-1} \frac{q}{n-t} \Big|\sum_{j=0}^{N-1}  \varepsilon_j \big(X^*_{2j}(g_{s+1,k}) - X^*_{2j}(g_{s,k}) \big) \Big| \Bigg]
                      &  \nonumber \\
& \leq    &  \frac{3 \cdot 2^6}{\gamma} 2^{-2k} \delta^{-2}\sum_{s=0}^{\check{S}-1}  E \Bigg[ \max_{(g_1,g_2) \in \cm_{k,s}}  \frac{q}{n-t} \Big|\sum_{j=0}^{N-1}  \varepsilon_j \big(X^*_{2j}(g_{2}) - X^*_{2j}(g_{1}) \big) \Big|  \Bigg]    
                      & \qquad  \textit{by \eqref{eq:inf-max}} \nonumber \\
&\leq     & \frac{3 \cdot 2^6 \, C_5}{\gamma} \   2^{-2k} \delta^{-2} \sum_{s=0}^\infty  \sqrt{\log \# \cm_{k,s}} \  \frac{2^{k-s} \delta \sqrt{q}}{\sqrt{n-t}} 
                      & \textit{by \eqref{eq:EW-max}} \nonumber \\
&\leq     & \underbrace{\frac{3 \cdot 2^6 \, C_5}{\gamma} C_6}_{=: C_7}  \ 2^{-2k} \delta^{-2} \, \sum_{s=0}^\infty  2^{s/2} 
2^{-k/2} n^{1/6} \frac{2^{k-s} \delta \sqrt{q}}{\sqrt{n-t}}
                      & \textit{by \eqref{eq:M}} \nonumber \\
&  =   &  2^{-3k/2} \underbrace{n^{1/2} (\log n)^{-1} (n-t)^{-1/2} q^{1/2} }_{\asymp (\log n)^{-1/2}} C_\delta^{-2} C_7  \sum_{s=0}^\infty 2^{-s/2}\nonumber\\
& \leq & C'' \ 2^{-3k/2} \ (\log n)^{-1/2} \nonumber
\end{IEEEeqnarray}
for all $k \in \Nbb$ and all but finitely many $n\in \Nbb$ if $C''>0$ is chosen sufficiently large. Of course, $2^{-3k/2} \leq 2^{-k}$ for all $k \in \Nbb$. Thus, there exists a number $n^{(P_3)} \in \Nbb$ such that for $n \geq n^{(P_3)}$ and all $k \in \Nbb$
\begin{IEEEeqnarray*}{rCl}
&         &   P \Bigg\{ \sup_{g \in \cg_k}  \sum_{s=0}^{\check{S}-1} \frac{q}{n-t} \Big|\sum_{j=0}^{N-1}  \varepsilon_j  \big(X^*_{2j}(g_{s+1,k}) - X^*_{2j}(g_{s,k}) \big) \Big| > 2^{2k-6} \gamma \delta^2 / 3 \Bigg\} \leq C' \ 2^{-k} (\log n)^{-1/2}.
\end{IEEEeqnarray*}

In conclusion, we can choose the constant $C_\delta$ independently of $n$ and $k$ such that
\begin{IEEEeqnarray*}{rCl}
 P_1 + P_2 + P_3  & \leq  & C \ 2^{-2k} \frac{q}{n-t} + e^M \ n^{-(k+1)} + \exp\Big( - C' \ 2^{k} \, n^{1/3} \Big) + C'' \ 2^{-k} (\log n)^{-1/2}
\end{IEEEeqnarray*}
for all $n \geq \max\big\{ n^{(P_1)},\, n^{(P_2)},\, n^{(P_3)}\big\}$ and all $k \in \Nbb$. 
\end{proof}

\section*{Acknowledgments}

MW thanks Florian Wechsung for revising the manuscript and giving some valuable comments.

\bibliographystyle{apalike}
\bibliography{note}  

\end{document}